\documentclass[11pt]{amsart}

\def\cal{\mathcal}
\newtheorem{theorem}{Theorem}
\newtheorem{lemma}{Lemma}
\newtheorem{conjecture}{Conjecture}

\def\ra{\rightarrow}
\def\qed{\hfill \mbox{\rule{0.5em}{0.5em}}}

\def\BP{{\Bbb P}}
\def\BR{{\Bbb R}}
\def\BN{{\Bbb N}}
\def\BE{{\Bbb E}}
\def\BZ{{\Bbb Z}}

\def\bone{\mbox{\bf 1}}

\def\p{\BP}
\def\e{\BE \,}
\def\ll{\log \! \log}
\def\l3{\ \log \! \log \! \log}
\def\lg4{\ \log \! \log \! \log \! \log}

\def\L{{\cal L}}

\def\X{{\cal X}}

\def\indist{\stackrel{\mbox{\small d}}{=}}
\def\Ak{A^{(k)}}
\newcommand{\ignore}[1]{}

\begin{document}
\title[Prime Factorization of a Uniform Random Integer]
{On the Amount of Dependence in the Prime Factorization
of a Uniform Random Integer}

\author[Richard Arratia]{Richard Arratia}
\address[Richard Arratia]{Univ.\ of Southern California \\
Department of Mathematics \\
Los Angeles CA 90089-1113}

\email{rarratia@math.usc.edu}

\thanks{
Lectures given June 26,27,29,30, 1998, at the
workshop on Probabilistic Combinatorics at the
Paul Erd\H{o}s Summer Research Center of Mathematics}

\date{March 18, 1999; updated September 4, 2000}
\maketitle

\begin{abstract}
How much dependence is there in the prime factorization of a
random integer distributed uniformly from 1 to $n$? How much
dependence is there in the decomposition into cycles of a random
permutation of $n$ points?  What is the relation between the
Poisson-Dirichlet process and the scale invariant Poisson process?
These three questions have essentially the same answers, with
respect to total variation distance, considering only small
components, and with respect to a Wasserstein distance,
considering all components.  The Wasserstein distance is the
expected number of changes -- insertions and deletions -- needed
to change the dependent system into an independent system.

In particular we show that for primes,
roughly speaking,
$2+o(1)$ changes
are necessary
and sufficient to convert a uniformly distributed random integer
 from 1 to $n$ into a random integer $\prod_{p \leq n} p^{Z_p} $
in which the multiplicity $Z_p$ of the factor $p$ is geometrically
distributed, with all $Z_p$ independent.
The changes are, with probability tending to 1, one deletion,
together with
a random number of insertions, having expectation $1+o(1)$.

The crucial tool for showing that $2+\epsilon$ suffices is a
coupling of the infinite independent  model of prime
multiplicities, with the scale invariant Poisson process on
$(0,\infty)$.
A corollary of this construction is the first metric bound on the
distance to the Poisson-Dirichlet in Billingsley's 1972 weak
convergence result.  Our bound takes the form:  there are
couplings in which
$$
\e \sum |\log P_i(n) - (\log n) V_i | = O( \log \log n),
$$
where $P_i$ denotes the $i^{th}$ largest prime factor and $V_i$
denotes the $i^{th}$ component of the Poisson-Dirichlet process.
It is reasonable to conjecture that $O(1)$ is achievable.
\end{abstract}

\tableofcontents

For the reader impatient to get to business, the main results are
given by
the bound
(\ref{dw upper claim}) in Theorem \ref{dw upper theorem},
and
the bound
(\ref{pd distance 2}) in Theorem \ref{thm for PD}.
A $ \$ 500 $ conjecture, related to Theorem  \ref{dw upper theorem},
is given by relation (\ref{500 question}), and a $ \$ 100 $
conjecture, related to Theorem \ref{thm for PD}, is given by
the bound (\ref{pd distance 1}).

\section{Lecture 1.  Growing a random integer}

I would like to thank the J\'anos Bolyai Mathematical Society and
the organizers of this conference for the honor of speaking here,
where the spirit of Erd\H{o}s seems so close.  At the conference
reception last night Imre Csiszar, student of R\'enyi, suggested that
Erd\H{o}s, now in heaven, has read the book where all best
proofs are given.  But I  prefer to believe that in heaven,
Erd\H{o}s by choice does not ask to see the book; he only asks of a proof,
``Is there an even better one in the book?''  And  he
watches
us fellow mathematicians here on earth, as we lecture and write and
discover; by not diving into the book, he can continue to compete
against us.  For the joy of discovery through one's own effort is
far more rewarding than even reading the book.
I would also like to thank my collaborators, Andrew D. Barbour and
Simon Tavar\'e, as their work and ideas pervade these lectures.
I hope that our
book on logarithmic combinatorial structures \cite{abt}, in preparation
since 1992, will soon see publication!

The guiding question for today's lecture is to ask, by analogy with
the Erd\H{o}s-R\'enyi notion of {\em growing} a random graph,
``can we grow a random integer?''
For guidance, we compare with the simpler task of ``growing'' a
random permutation.

NOTATION:
$n$ is always the parameter, rather than the random object.  We
consider

\hspace{1in} -- a permutation chosen uniformly from ${\cal S}_n$

\hspace{1in} -- an integer $N$ chosen uniformly from 1 to $n$.

We may emphasize the role of the parameter $n$ by writing it
explicitly:
$$
\p_n(N=i) \equiv \p(N(n)=i) \  = 1/n \ \mbox{ for } i=1,2,\ldots,n.
$$
For the prime factorization we write
$$
N(n) = \prod p^{C_p(n)}, \mbox{ so } (C_p(n))_p =
(C_2(n),C_3(n),C_5(n),\ldots)
$$
is a dependent process, with $C_p(n)$ identically zero if $p>n$.

The baby fact: as $n \rightarrow \infty$
\begin{equation}
\label{primes limit}
(C_p(n))_p \Rightarrow (Z_p)_p = (Z_2,Z_3,Z_5,\ldots)
\end{equation}
with independent, geometrically distributed coordinates,  with
$\p(Z_p \geq~k)=p^{-k}$.

\noindent{\bf Proof.}  Given $j$ distinct primes $p_1,\ldots,p_j$,
and integers $i_1,\ldots,i_j\geq 0$, write
$d=p_1^{i_i}\cdots p_j^{i_j}$ so that as events,
$$
\{ C_{p_1} \geq i_1,\ldots, C_{p_j}\geq i_j \}
  = \{  \ d | N \ \}.
$$
Thus
\begin{equation}
\label{eq1.1}
\p_n( \ d|N\ ) = \frac{1}{n} \big\lfloor \frac{n}{d} \big\rfloor \
 \rightarrow \frac{1}{d} = \p(Z_{p_1} \geq i_1,\ldots,Z_{p_j} \geq
 i_j),
\end{equation}
and $j$-dimensional differencing yields
\begin{equation}
\label{eq1.2}
\p_n(C_{p_1}=i_1,\ldots,C_{p_j}=i_j) \ \rightarrow
\p(Z_{p_1}=i_1,\ldots,Z_{p_j}=i_j).
\end{equation}\hfill \qed

In fact, the approximation error in (\ref{eq1.1}) is at most
$1/n$, so the error
in (\ref{eq1.2}) is at most $2^j/n$ --- a crude upper bound which
comes from taking the absolute value inside.
The sum over all integers  $d = \prod_{p \leq b} p^{a_p}$ whose largest prime
factor is at most $b$, i.e.
\begin{equation}\label{dtv def 1}
\sum_{d:\, P^+\!\!(d) \leq b}
|\  \p(C_p(n)=a_p \ \forall p \leq b) -\p(Z_p=a_p \ \forall p \leq
b)|,
\end{equation}
gives the total variation distance $d_{TV}(b,n)$, restricting to primes not
exceeding $b$, and surprisingly, the crude upper bound $u(b,n)$, formed by  taking
absolute values inside, gives pretty good information.  But the
information is not nearly as good as the result of Kubilius; we will
describe these bounds later, in sections \ref{dtv section} and
\ref{absolute section}. \vspace{1pc}

Since $\e Z_p = \frac{1}{p-1}$ with  $\sum \e Z_p =
\infty$, it follows from the independence of the $Z_p$ that
$$
1=\p\left( \sum Z_p = \infty\right).
$$
Thus we call the multiset, having $Z_p$ copies of $p$ for each
prime, the \emph{natural random infinite multiset of primes}.  The
guiding
question for today's lecture, can we grow random integers, is now
stated more precisely as:  can we construct $N(1),N(2),\ldots,N(n),\ldots$
all on a single probability space, together with the natural
random infinite multiset of primes, so that the $C_p(n)$
{\em evolve smoothly}, with $C_p(n) \rightarrow Z_p$
as $n \rightarrow \infty$.  The answer is yes in a sense; the construction
culminating with \eqref{construct N} at the end of Lecture 3 has
$C_p(n) \rightarrow Z_p$ \emph{in probability}, but with probability one,
$\liminf C_p(n)=Z_p$ and $\limsup C_p(n)=Z_p+1$.

\subsection{Overview: the limits for primes, small and large}

This section presents the material from one of the two
transparencies which were shown repeatedly throughout the lectures;
the material from the other transparency  appears in this writeup
as Conjecture \ref{500 conjecture} in section \ref{score} and
as Conjecture \ref{100 conjecture} in section \ref{distance}.

$N(n)$ is uniform 1 to $n$, with \hfill to focus on---

$N(n) = \prod p^{C_p(n)}$ \hfill   \emph{---small factors},

$N(n) = P_1(n)P_2(n)\cdots, P_1\geq P_2 \geq \cdots$, prime or
1, \hfill  \emph{---large factors}
\vskip 1pc

\noindent{\bf LIMITS} in distribution as $n \rightarrow \infty$

$(C_p(n))_p \Rightarrow (Z_p)_p,$ \hfill
  independent Geometric

$$\left({\log P_1(n) \over \log n},{ \log P_2(n) \over \log n} ,\ldots \right)
\Rightarrow (V_1,V_2,\ldots), \ \ \ \ \ \mbox{Poisson-Dirichlet}
$$
proved by  Billingsley in 1972,
where the limit is now known as the Poisson-Dirichlet process,
with parameter 1.
\vskip 1pc

\noindent{\bf HOW CLOSE?}  One coupling has
\begin{equation}\label{summary a}
\e \sum_{p \leq n} |C_p(n)-Z_p| \ \leq 2 + O\left( \frac{(\log
\log n)^2}{\log n} \right),
\end{equation}
\begin{equation}\label{summary b}
  \e \sum |\log P_i(n) -(V_i) (\log n)  | = O(\log \log n).
\end{equation}
We note that in (\ref{summary a}), $2-\epsilon$ is not possible,
while in (\ref{summary b}), $O(1)$ should be possible.
The reason that $\ll n$ appears in our bound for
(\ref{summary b}) is made clear by equation
(\ref{log log asymptotics}) in conjunction with the
proof of Theorem \ref{thm for PD}.

\subsection{Sketch of the coupling to grow $N(n)$}\label{sketch}

Take a size biased permutation, say $Q_1,Q_2,\ldots,$
of the prime factors in $2^{Z_2}3^{Z_3}\cdots$\ .

Let $J(n) =$ the largest partial product $Q_1Q_2\cdots Q_j \leq n$.
Write $L \equiv L(n)$ for the number of factors, so that $J(n)=Q_1 Q_2 \cdots
Q_{L(n)}$.

Show: $J(n)$ has approximately the same distribution as $H(n)$, where by definition
$H(n)$ has the harmonic distribution on $[n]$,
\begin{equation}
\label{harmonic}
\p(H(n)=i) = \ \frac{1/i}{1+ \ 1/2 + \cdots + 1/n} \ = \frac{1}{i \, h_n}, \ \
i=1,2,\ldots,n.
\end{equation}

Fill in one extra factor, $P_0(n)$, to be prime or one.  Use a
random uniform $U \in (0,1]$ to choose
uniformly from the $1+\pi(n/J(n))$ possibilities with
$J(n) P_0(n) \leq n$.

Show that the resulting random integer, $J P_0 \equiv J(n) P_0(n)$
is close to uniform.  In fact, from Lemma \ref{lemma j bound}
\begin{equation}
\label{harmonic bound}
d_{TV}(J(n),H(n) \, ) = O \left( \frac{1}{\log n}\right),
\end{equation}
and from Lemma \ref{lemma jp bound}
\begin{equation}
\label{JP versus N}
d_{TV}(J(n)P_0(n),N(n) \, ) = O \left( \frac{\log \log n}{\log n}\right).
\end{equation}

Finally, modify the coupling on the event whose probability is the
left side of (\ref{JP versus N}), so that the modified versions of
$J P_0$ are exactly uniform, for all $n$.

Noga Alon asked, ``Does $\p(J(n)P_0(n)=1) \sim 1/n$?  The total
variation distance bound (\ref{JP versus N}) does not determine
the
answer.  Since $\p(P_0(n)=1 \ | \ J(n)=1) =
 1/(1+\pi(n)) \sim \log n / n$,
Alon's question is equivalent to, ``Does $\p(J(n)=1) \sim 1/\log
n$?  Now if $Z_p=0$ for all $p \leq n$ then we must have $J(n)=1$,
and $\p(Z_p=0 \ \forall p \leq n)$
$=\prod_{p \leq n} (1- \, 1/p)$
$\sim e^{-\gamma}/\log n$ by Mertens' Theorem.  Thus Alon's question is
equivalent to asking:  does $(1-e^{-\gamma})/\log n$ give the
asymptotic  probability that \emph{there are} one or
more primes less than or equal to $n$ in the infinite multiset
{ \em and, in the size biased permutation, }
some prime greater than $n$ comes before all of them.
The answer, which I did not give at the workshop, but was
more or less
evident from (\ref{compare harmonic}) and (\ref{simpler 1}),
is yes; see Lemma \ref{lemma j bound}  for further details.  While this
affirmative answer might give us hope that $J(n)P_0(n)$ is close
to uniform in the sense that for all $i\leq n$,
$\p(J(n)P_0(n)=i) \sim 1/n$, it is clearly not so.
In fact, (for the coupling in Lecture 3, which is slightly different
from the coupling described in this lecture,)
for fixed $i$, having $\omega(i)$ distinct prime factors, as
$n \rightarrow \infty$, $\p(J(n)P_0(n)=i) \sim (1+\omega(i))/n$;
see (\ref{ph exact}) and (\ref{ph simple}) for details.
Thus $i=1$ is the only fixed integer
having the correct asymptotic probability!

\vskip 2pc
\noindent
How does $J(n)P_0(n)$ evolve as $n$ grows?  Write
$p^{\#}$ for the smallest prime larger than $p$,
with $1^{\#}=2$.
The following properties
hold for all $n>1$ and for all outcomes.

$J(n) P_0(n) \in [1,n]$.

\noindent Each factor has its
own smoothness:

 $J(n)/ J(n-1)$ is one or a prime,

 $P_0(n) = P_0(n-1), P_0(n-1)^{\#} $, or 1,

\noindent and the jumps in the two factors are linked:

$J(n) \neq J(n-1)$ implies
  ( $J(n)=n$ and  $P_0(n)=1$),

$P_0(n)<P_0(n-1)$ implies
  ( $J(n)=n$ and  $P_0(n)=1$).

\subsection{Size biased permutations}\label{size bias section}

Several in the audience requested clarification:  what is
a size biased permutation?  Given $k$ objects, with ``sizes'' or
``weights''
$r_1,r_2,\ldots,r_k>0$, we can carry out
a size biased permutation using $k$ independent standard exponentially
distributed ``alarm clocks'' $S_1,S_2,\ldots,S_k$  with $\p(S_i > t)
= e^{-t}$ for all $t>0$.
The labels $W_i := S_i/r_i$ are exponentially distributed with
$\p(W_i > t) = e^{-r_i t}$, (we say ``rate $r_i$'' or ``mean $1/r_i$,'')
and the labels are independent, with
$\p(W_i=W_j)=0$ for every $i \neq j$.  The ranking of the labels
induces a size biased permutation of the objects.  Observe that
 $\p(W_i$ is the smallest among $W_1,W_2,\ldots,W_k)$
$= r_i/(r_1+r_2+\cdots+r_k)$, which is a familiar fact from the
study of finite-state Markov chains in continuous time,
where the $r_i$ are jump rates.  This calculation
shows the distribution of the first item; iterating and
using the memoryless property of exponentials gives a
product formula for the distribution of the full size biased permutation.

For primes,
the size of $p$ is $\log p$, and such a size biased permutation was used
in the 1984 Ph.D. thesis of Eric  Bach
\cite{bach}
to generate uniformly distributed random integers, factored into
primes,
and independently by
Donnelly and Grimmett in 1993
\cite{DG}
to give a simple proof of Billingsley's Poisson-Dirichlet limit
for prime factors -- discussed at (\ref{Billingsley pd}) below.
In our context, the infinite random multiset of primes, there
are infinitely many labels, but with probability one, the only
limit point is zero, and there is a largest label.  Our
size biased permutation starts with the prime having this largest
label; listing the labels from largest downward tends to put smaller
primes toward the front of the list.

\subsection{An example of the growth of an integer}

Take for example an outcome of
the experiment with
\begin{equation}
\label{prime example counts}
Z_2=3,Z_3=1,Z_5=0,Z_7=1,Z_{11}=1,\ldots
\end{equation}
and size biased permutation
\begin{equation}
\label{prime example}
   Q_1,Q_2,\ldots,Q_6,\ldots = 3,2,2,11,2,7,\ldots.
\end{equation}

The first seven partial products, one through
$3 \cdot 2 \cdot 2 \cdot 11 \cdot 2 \cdot 7$,
shown on a logarithmic scale, are

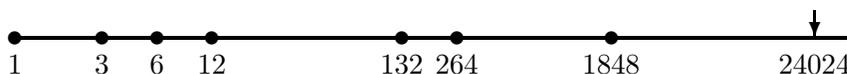
\begin{picture}(500,50)(-20,0)
\thicklines
\put(0,25){\line(1,0){320}}
\put(0,25){\circle*{5}}
\put(32.96,25){\circle*{5}}
\put(53.75,25){\circle*{5}}
\put(74.54,25){\circle*{5}}
\put(146.48,25){\circle*{5}}
\put(167.27,25){\circle*{5}}
\put(225.65,25){\circle*{5}}
\put(302.6,35){\vector(0,-1){10}}

\put(0,15){\makebox(0,0){1}}
\put(32.96,15){\makebox(0,0){3}}
\put(53.75,15){\makebox(0,0){6}}
\put(74.54,15){\makebox(0,0){12}}
\put(146.48,15){\makebox(0,0){132}}
\put(167.27,15){\makebox(0,0){264}}
\put(225.65,15){\makebox(0,0){1848}}
\put(302.6,15){\makebox(0,0){24024}}
\end{picture}

The arrow pointing to 24024 = 1848 13 is to show that the next
partial product will lie on or to the right of this location.
We know this because the information $6=\sum_{p < 13} Z_p $,
together with the good luck that the first six primes in the size
biased permutation are less than 13, implies that the seventh
prime in the size biased permutation will be 13 or greater. In
particular,
for $1848 \leq n < 13 \times 1848,$   we know that $J(n)=1848$, even though we
don't know the seventh prime in the size biased permutation, accounting for
the last line of the table below.

The jumps in $J(n)$ are shown by double horizontal lines in the
two tables below.

\begin{flushleft}

\hspace{.6in}\begin{tabular}{|c|c|c|c|}
\hline
  $n$ & $J(n)$ & $1+\pi(n/J(n))$ & $P_0(n)$ \\ \hline \hline
1 & 1 & 1 &  1 \\ \hline
2 & 1 & 2 & 1 if $U\leq.5$, 2 if $U>.5$ \\ \hline \hline
3,4,5  & 3 & 1 & 1 \\ \hline \hline
6 to 11 & 6 & 1 & 1 \\ \hline \hline
12 to 23 & 12 & 1 & 1 \\ \hline
24 to 35 & 12 & 2 & 1 if $U\leq.5, 2 $ if $U>.5$ \\ \hline
36 to 59 & 12 & 3 & 1 if $U\leq1/3, \ldots, 3$ if $U>2/3$ \\ \hline
60 to 83 & 12 & 4 & 1 if $U\leq1/4, \ldots, 5$ if $U>3/4$ \\ \hline
84 to 131 & 12 & 5 & 1,2,3,5, or 7 \\ \hline \hline
132 to 263 & 132 & 1 & 1 \\ \hline \hline
$[264,2 \cdot 264)$ & 264 & 1 & 1 \\ \hline
$[2 \cdot 264, 3 \cdot 264)$ & 264 & 2 &  1 if $U\leq.5, 2 $ if $U>.5$ \\ \hline
$[3 \cdot 264, 5 \cdot 264)$ & 264 & 3 &1 if $U\leq1/3, \ldots, 3 $ if $U>2/3$ \\  \hline
$[1320,1848)$ & 264 & 4 & 1, 2, 3, or 5 \\ \hline \hline
$[1848,13 \cdot 1848)$ & 1848 & 1,2,3,4,5, or 6 & 1,2,3,5,7, or 11\\ \hline
\end{tabular}
\end{flushleft}
\vspace{2ex}

To continue the above example, we will consider three cases: the
first being $U \in (.5,.6]$, the next being $U \in (.6,2/3]$, and
the last being $U > 5/6$.

\begin{center}
\begin{tabular}{|c || c | c || c|c ||c| c |} \hline
 &  \multicolumn{2}{c|}{ $.5 < U \leq .6$}
 &  \multicolumn{2}{c|}{ $.6 < U \leq 2/3 $}
 &  \multicolumn{2}{c|}{ $5/6 < U$} \\ \hline
n & $P_0(n)$ & $J(n)P_0(n)$& $P_0(n)$ & $J(n)P_0(n)$&
$P_0(n)$ & $J(n)P_0(n)$  \\ \hline

1 & 1 & 1 & 1 & 1 & 1 & 1 \\ \hline
2 & 2 & 2 & 2 & 2 & 2 & 2 \\ \hline \hline
3,4,5 & 1 & 3 & 1 & 3 & 1 & 3 \\ \hline \hline
[6,11) & 1 & 6 & 1 & 6 & 1 & 6 \\ \hline \hline
[12,23) & 1 & 12 & 1 & 12 & 1 & 12 \\ \hline
[24,35) & 2 & 24 & 2 & 24 & 2 & 24 \\ \hline
[36,59) & 2 & 24 & 2 & 24 & 3 & 36 \\ \hline
[60,83) & 3 & 36 & 3 & 36 & 5 & 60 \\ \hline
[84,131) & 3 & 36 & 5 & 60 & 7 & 84 \\ \hline \hline
[132,263) & 1 & 132 & 1 & 132 & 1 & 132 \\ \hline \hline
[264,528) & 1 & 264 & 1 & 264 & 1 & 264 \\ \hline
[528, 792) & 2 & 528 & 2 & 528 & 2 & 528 \\ \hline
[792, 1320) & 2 & 528 & 2 & 528 & 3 & 792 \\ \hline
[1320,1848) & 3 & 792 & 3 & 792 & 5 & 1320 \\ \hline \hline

\end{tabular}
\end{center}

Recall that
$d_{TV}(J(n)P_0(n),N(n))\rightarrow 0$, but the random integer we
grow is {\em not exactly uniform}.  For comparison, random
permutations have similar behavior and can be grown exactly.
Eric Bach's procedure gets an {\em exactly uniform} random
integer,
but not with $n$ {\em evolving}.

\subsection{Notions of distance}

\subsubsection{\bf The expected number of insertions and deletions
needed}\label{indel section}

The material below on $d_W$ was delivered at the workshop at the start of
Lecture 2; the material on $d_{TV}$ was drawn out in workshop conversations
with Joel Spencer, and I prepared a transparency for the lecture
but did not deliver it, for lack of time.

We consider the metric $d$ on positive integers
which counts the number of changes needed to convert the prime
factorization of one integer into that of the other.
For example, $d(40,500) = d(2^3 5^1,2^2 5^3)=3$, $d(8,3)=4$, and
$d(i, ip) =1$ for
any integer $i$ and prime $p$.  Writing $(i,j)$
for the greatest common divisor, and $\Omega(i) $ for the number
of prime factors, including multiplicities, we have in general
$$
d(i,j) = \Omega\left(\frac{i}{(i,j)}\right)+
\Omega\left(\frac{j}{(i,j)}\right).
$$
If we think of converting $j$ to $i$, the first term above is the
number of {\em insertions} needed, and the second term is the
number of {\em deletions}, so we think of $d$ as the
insertion/deletion distance, analogous to the string edit distance
of Levenstein \cite{levenstein} or Ulam \cite{ulam}; see
the book of Kruskal and Sankoff \cite{time warps} for more history.

For the sake of comparing the uniform random $N(n)$ with the
infinite random multiset of primes, clearly primes $p>n $ should
not be considered.  Thus, we code up the relevant part of the multiset
by {\em defining}
\begin{equation}
\label{def M}
M(n) := \prod_{p \leq n} p^{Z_p}, \ \ \ \mbox{ with }
Z_2,Z_3,\ldots \mbox{ independent geometric}.
\end{equation}Recall that
$$
N(n) = \prod_p p^{C_p(n)} = \prod_{p \leq n} p^{C_p(n)} \ \ \mbox{
is uniform 1 to } n.
$$

The insertion-deletion distance between these two random integers
is
a random, nonnegative integer
$$
d(N,M) = \Omega\left( \frac{NM}{(N,M)^2} \right) =
\sum_{p \leq n} |C_p(n) - Z_p |.
$$
The Wasserstein distance between two random objects $M$ and $N$,
for a given
metric $d$, is by definition the infimum, over all couplings, of
the expected value of $d(N,M)$.  Recall that a coupling means a
construction of $M$ and $N$ simultaneously on a single probability
space; it is understood that the {\em marginal distributions} for $M$
and for $N$ have been specified in advance, but there is no other
constraint on their {\em joint distribution}.  A compactness
argument shows that the inf is achieved; see for example
\cite{dudley}. To emphasize the role of the parameter $n$, which
determines the marginal distributions of $N(n)$ and $M(n)$, we
define
\begin{equation}\label{def dw}
d_W(n) := \min_{couplings} \e \ d(N(n),M(n)).
\end{equation}

Our result is
\begin{equation}
\label{dw=2 claim}
 \lim_{n \rightarrow \infty} d_W(n) \ \ =2.
\end{equation}
We will prove the hard part of this, that $\limsup \ d_W(n) \leq 2$,
in Lecture 3, essentially by analyzing the growth of a random
integer.  The matching lower bound, from \cite{primedw}, is that
$\liminf \ d_W(n) \geq 2$; this is relatively easy, and the reader
is challenged to discover a proof, before considering the hint
given by the paragraph following (\ref{dw upper claim}).  A related
discussion appears in section 22 of \cite{dimacs}.

For perspective on the content of (\ref{dw=2 claim}), we
note that even the bound  $d_W(n) = O(1)$ is very strong.  For
instance, by comparing $(C_p(n))_{p \leq n}$
with the independent process $(Z_p)_{p \leq n}$,
the following consequences can be derived easily: (see \cite{primedw})

$d_W(n) = o(\log \log n)$  implies the Hardy-Ramanujan Theorem for
the normal order of the number of prime divisors.

$d_W(n) = o(\sqrt{\log \log n})$  implies the Erd\H{o}s-Kac
Central Limit Theorem.

$d_W(n) = O(1)$ gives another proof
of the ``conjecture of LeVeque,'' that the error in the
Central Limit Theorem is
$O(1/\sqrt{\log \log n})$; the first proof was given by  R\'enyi and
Tur\'an in 1957 \cite{renyi turan}.

$d_W(n)=o(\l3 n)$  yields that the optimal rate in the simplest case of
the
Brownian motion convergence of  Billingsley and Philipp,
\cite{billingsley-monthly,billingsley,philipp}, for the expected
sup norm, is order of $\l3 n /\sqrt{\log \log n} $, and no smaller.
The underlying idea is that for coupling
Brownian motion with the rate one (centered) Poisson process, with both
processes
run until the variance is $t$, and \emph{without} rescaling by
$\sqrt{t}$,
the coupling distance grows on the order of $\log t$.
For primes this is applied with $t := \sum_{p \leq n} 1/p \ \sim \log \log
n$.  See
Kurtz (1978) \cite{kurtz} for the upper bound,
Rio (1994) \cite{rio2} for the lower bound, and
\cite{primedw} for the connection with primes.

\subsubsection{\bf The total variation distance}\label{dtv section}

This section gives a ``one transparency'' overview of the
situation involving total variation distance for primes.

NOTATION:  $\beta \in [0,1]$, fixed or $\beta=\beta(n) \rightarrow 0; \
u \equiv 1/\beta$.
\begin{equation}\label{dtv beta}
  d_{TV} := d_{TV}(n^\beta,n) := \ \min_{couplings}
  \p\left( \ (C_p(n))_{p \leq n^\beta} \neq (Z_p)_{p \leq n^\beta}
  \right).
\end{equation}
Kubilius \cite{kub} in the 1950's showed a) below with an upper bound
of the form $\exp(-c u)$, Barban and Vinogradov \cite{barban}
improved this to the form $\exp(-c u \log u)$, and Elliott
\cite{elliott} gave the particular constants in b) below.

a) If $\beta \rightarrow 0$ then $d_{TV} \rightarrow 0$, and

b) $d_{TV} = O\left( \exp(-\frac{1}{8} u \log u) + n^{-1/15}
\right)$.

Elliott  \cite{elliott} gives a partial converse: if $\beta \rightarrow
1$ then $d_{TV} \not\rightarrow 0$,  and it is not hard to see the full
converse, that $d_{TV} \rightarrow 0$ implies $\beta \rightarrow
0$. In Spring 1996 I stated a conjecture: that for fixed $\beta$,

c) $H(\beta) := \lim_{n \rightarrow \infty} d_{TV}$  \ exists, and is the
same as the limit for permutations.  This limit is given explicitly
as an integral in \cite{tv1}, and
proved to be the limit for permutations in Stark's
1994 PhD Thesis; see \cite{stark th}.

Later in 1996 Tenenbaum \cite{crible} improved b) to
$$
d_{TV} = O\left( \exp( - u(\log u + \log \log u -
(1+\log {2}+\epsilon))) + n^{\epsilon-1} \right)
$$
and Arratia and Stark \cite{AS} proved c).
Soon after Tenenbaum \cite{crible} also proved c), with a rate.
The limit was further identified \cite{T1} as a distance between
the restrictions to $[0,\beta]$ of two processes which will be discussed in
lectures 3 and 4,  the
Poisson-Dirichlet process with parameter 1, and the
scale invariant Poisson process with intensity $(1/x) \ dx$:
for every $\beta \in [0,1]$,
\begin{equation}\label{H exact}
H(\beta)= d_{TV} \left( \{V_i: V_i < \beta\}, \{X_i: X_i < \beta \} \right)
\end{equation}
$$
=\min_{couplings} \p \left( \{V_i: V_i < \beta\}\neq \{X_i: X_i < \beta \} \right).
$$

\subsubsection{\bf The bound from taking absolute values inside}
\label{absolute section}

The crude procedure of ``taking the absolute values inside''
described following (\ref{dtv def 1}) shows that the total variation
distance $d_{TV}(n^\beta,n)$ is at most $u(n^\beta,n)$ where
\begin{equation}\label{def u}
  u(b,n)={1 \over n} \sum_{d \geq 1:\, P^+\!(d) \leq b} \left\{ n \over
d \right\} 2^{\omega(d)}.
\end{equation}
The notation here is  $\{ x \}$ for the fractional part of $x$,
 $\omega(d)$
for the number of distinct prime factors of $d$, and
$P^+(d)$ for the largest prime factor of $d$, with
$P^+(1)=1$.
Analysis of $u(b,n)$ (see \cite{primedw}) shows that
when $b,n \rightarrow \infty$ together,
the  threshold
for whether $u(b,n)$ tends to zero or infinity
is
$\log b = (.5 \pm \epsilon) \log n \l3 n / \ll n$. In fact
if $
\log b \leq {\log n \l3 n / (c \ll n)}$ with  $c>2+a > 2$,
then $u(b,n) = o \left( (\log n)^{-a} \right),$
while if $\log b >
{\log n \l3 n  / (c \ll n)}$, for  $c < 2$,
then  $u(b,n) \ra \infty$.
While much weaker than the upper bound  of Kubilius, the
upper bound $u(b,n)$ on the total variation distance
has the virtue
that it also serves as an upper bound on the Wasserstein
distance: for all $1 \leq b \leq n$,
\begin{equation}\label{dw and dtv}
  d_{TV}(b,n) \leq d_W(b,n) \leq u(b,n),
\end{equation}
where $d_W(b,n)$ is the minimum expected number of insertions and
deletions needed to convert $\prod_{p \leq b} p^{C_p(n)}$ to
$\prod_{p \leq b} p^{Z_p}$. [The $d_W(n)$   in (\ref{def dw})
is the special case $b=n$ of this, i.e. $d_W(n)=d_W(n,n)$.]
While we know the limit of
(\ref{H exact}) for $d_{TV}(n^\beta,n)$ as $n \rightarrow \infty$
with $\beta \in (0,1)$ fixed, the corresponding behavior of $d_W(b,n)$
remains unknown.

\noindent{\bf Open question}  What is the limit of
$d_W(n^\beta,n)$ as $n \rightarrow \infty$ for fixed $0<\beta<1$?

We know that for $\beta=1$, the limit is 2 --- this is (\ref{dw=2 claim}).
For all $\beta \in [0,1]$,
it can be seen from the coupling in Lecture 4
that $\limsup d_W(n^\beta,n) \leq 2 \beta$ --- but only for
$\beta=1$ is there a matching lower bound.
The natural candidate for the limit of the distance is the
Wasserstein distance for
the limit systems,
\begin{equation}\label{HW exact}
H_W(\beta) := d_{W} \left( \{V_i: V_i < \beta\}, \{X_i: X_i < \beta \} \right)
\end{equation}
$$
\equiv
\min_{couplings} \e \left| \{V_i: V_i < \beta\}\Delta \{X_i: X_i < \beta \} \right|.
$$
Thus, the open question involves two tasks. First, prove
(or disprove!) that the
limit exists and equals $H_W(\beta)$.

\noindent {\bf Conjecture 0.}
$ \forall \beta \in [0,1], \ H_W(\beta) = \lim d_W(n^\beta,n).
$
\\ \\
The second task in our open question is to find an explicit
formula for $H_W(\beta)$. We know only $H_W(1)=2,$ $H_W(\beta) \leq 2 \beta,$
and trivially, $H_W(0)=0$ and $H_W$ is monotone.

\section{Lecture 2. Growing a random permutation}

This is a fresh start --- the following can be understood easily from
scratch, without worrying about the connection with prime
factorizations.  However, our example for permutations below has been
carefully cooked to match the example for primes in Lecture 1.

Write $C_i(n)$
for the number of cycles of length $i$ in our random permutation
$\pi \in {\cal S}_n$, so that always $C_i(n)=0$ if $i>n$, and
$n = \sum i C_i(n)$.
The analog of Theorem \ref{primes limit}, provable easily with
inclusion-exclusion,  is that
as $n \rightarrow \infty$
\begin{equation}
\label{perms limit}
(C_i(n))_i \Rightarrow (Z_i)_i = (Z_1,Z_2,Z_3,Z_4,\ldots)
\end{equation}
with independent coordinates $Z_i$, Poisson distributed with
$\e Z_i = 1/i$.  (Note, we are deliberately
re-using the same notation that we used for prime factorizations,
although the index
must be a prime in the latter case.  To contrast the
two situations,  $Z_p$ is either geometrically distributed, with
$\p(Z_p \geq k)=p^{-k}$
and $\e Z_p = 1/(p-1)$, or else $Z_p$ is Poisson, with
$\e Z_p = 1/p$.)  That there is an independent process limit, without
rescaling, for both prime factorizations of random integers,
and the cycle structure of random permutations, is the most
basic ingredient in the analogy between these two structures.

Historical notes:  the similarity between primes and permutations
appears to have been noted first in 1976 by Knuth and Trabb Pardo
\cite{knuth}. The similarity they note is based on the common
limit behavior for the $i^{th}$ largest prime factor and the
$i^{th}$ longest cycle; further aspects of the similarity,
are discussed in \cite{AMS1}, which sets forth the role of
``conditioning'' the independent limit process as the
fundamental reason for this similarity. In contrast, in these
lectures we will focus on the closeness of primes and permutations
to the scale invariant Poisson process, with its relations involving
spacings and size biased permutations, as the underlying reason
for the similarity.  A reasonable attribution for the
construction below, based on canonical cycle notation, is Feller
1945 \cite{feller45}, but the explicit connection with cycle
\emph{lengths} seems to appear first in the unpublished lecture notes
\cite{dp86}; it was independently discovered in \cite{statsci2}
and elaborated on in \cite{tv3}, which attached the name ``Feller
coupling'' to this construction.

Start with the canonical cycle notation for a permutation $\pi$.
For
example, the permutation $\pi$  with $1 \mapsto 5, 2 \mapsto 2,$
$3 \mapsto 1,$ $4 \mapsto 4$, $5 \mapsto 3$, $6 \mapsto 7$, $7 \mapsto 6$
is written as $\pi = (1 5 3)(2)(4)(6 7)$.  In writing the
canonical cycle notation for a random $\pi \in {\cal S}_7$,
one always starts with ``(1 \ '', and then makes a seven-way
choice, between ``(1)(2\  '', ``(1 2 \ '', $\ldots$, and  ``(1 7 \ ''.
One continues with a six-way choice, a five-way choice, $\ldots$,
a two-way choice, and finally a one-way choice.

Let $\xi_i$ be defined as the indicator function $\xi_i = 1($
close off a cycle when there is an $i$-way choice).
Thus
$$
\p(\xi_i=1) = \frac{1}{i}, \ \p(\xi_i=0) = \frac{i-1}{i}, \
\mbox{and } \xi_i,\xi_2,\ldots,\xi_n \mbox{ are independent}.
$$
An easy way to see the independence of the
$\xi_i$ is to take $D_i$
chosen from 1 to $i$ to make the $i$-way choice, so
that $\xi_i=1(D_i=1)$. Absolutely no computation is needed to verify
that the map constructing canonical cycle notation,
$(D_1,D_2,\ldots,D_n) \mapsto \pi$,
 from
$[1] \times [2] \times \cdots \times [n]$ to ${\cal S}_n$,
is a bijection.  The ``decision'' variables $D_1,D_2,\ldots,D_n$
determine the random permutation on $n$ points, while the
Bernoulli variables $\xi_1,\xi_2,\ldots,\xi_n$ determine the cycle
structure, and something more
--- a size biased permutation SB($n$) of the cycle lengths!
The total number of cycles is
$$
K_n := \# \mbox{ cycles} = \xi_1+\xi_2+\cdots+\xi_n,
 \mbox{ with } \e K_n = 1 + \frac{1}{2}+\cdots + \frac{1}{n} =
 \gamma + \log n + o(1).
$$

[Incidentally, the relations above give a quick and dirty, but
pretty, way to see that the entropy $h((C_i(n))_i)$
for the cycle structure
of a random permutation of $n$ objects is asymptotically
$(\log n)^2/2$.  Namely,
the entropy of a size biased permutation of $k$ objects
is at most the entropy of a \emph{uniform}
permutation of $k$ objects, which  is $\log k! \leq k \log k$.
Thus the entropy $h$(SB($n$)) of a size biased permutation of the
$K_n$ cycle lengths of a random $n$-permutation is at most
$\e K_n \log K_n$ $\sim \log n \ll n$.  The entropy of the
Bernoulli($1/i)$ random variable $\xi_i$ is
$h(\xi_i) = -(\frac{1}{i} \log \frac{1}{i} +\frac{i-1}{i} \log \frac{i-1}{i}) $
$=\log i -\log(i-1) + \frac{1}{i}\log(i-1)$ for $i\geq 2$,
so
$$
\sum_1^n h(\xi_i) =
 \sum_2^n (\log \frac{i}{i-1}+\frac{1}{i}\log (i-1))
= \log(n+1) + \sum_1^{n-1} \frac{\log i}{i+1} \
 \sim \frac{(\log
n)^2}{2}.
$$
Since the cycle structure together with the size biased
permutation of the cycle lengths determine $\xi_1,\ldots,\xi_n$,
we have  $(\log n)^2/2 \sim \sum_1^n h(\xi_i) =$
$h((C_i(n))_i) +  h$(SB($n))$ $=h((C_i(n))_i)+ O(\log n \ll n) $,
hence $h((C_i(n))_i) \sim (\log n)^2/2$.  As an exercise,
the reader can try to verify this directly from
Cauchy's formula, and perhaps (open problem?) give an asymptotic expansion!]

The coupling which ``grows'' a random permutation
requires a very simple idea: {\em for all } values of $n$, use
the same $D_1,D_2,\ldots,$ and hence the same
$\xi_1,\xi_2,\ldots$ \ .

The Feller coupling, as motivated by the process of
writing out canonical cycle notation, ``reads'' $\xi_1\xi_2\cdots
\xi_n$ from right to left: the length of the first cycle is
the waiting time to the first one, the length of the next cycle is
the waiting time to the next one, and so on.  The multiset of
cycle lengths can be determined without regard to right or left:
every
$i$-spacing in $1 \xi_2 \xi_3 \cdots \xi_n 1$, that is,
every pattern of two ones separated by $i-1$ zeros, corresponds to a cycle of
length $i$.   The spacing from
the rightmost one in $1 \xi_2 \xi_3 \cdots \xi_n$ to the
``artificial'' one at position $n+1$  corresponds to the first cycle in canonical
cycle notation, and also to the factor $P_0$
in our construction for growing an almost uniform random integer
$J(n)P_0(n)$.

Since we are
interested in growing with $n$, we will always read
$\xi_1\xi_2\xi_3\cdots$ from left to right.
Recall that $\xi_1=1$ identically, so $\xi_1\xi_2\xi_3\cdots$ is
a random infinite word in the alphabet \{0,1\}, starting with a 1.
Almost surely, this sequence has infinitely many ones, since the
$\xi_i$ are independent with $\sum_{i \geq 1} \e \xi_i$
$=\sum_{i \geq 1} 1/i= \infty$.  We {\em define} the
inter-one spacings  $B_1,B_2,\ldots \in \BN$ by the
requirement that
$$
\xi_1\xi_2\xi_3\cdots \mbox{ has ones at } 1, 1+B_1,
1+B_1+B_2,\ldots,\ \mbox{ and nowhere else} .
$$
The example which matches the example (\ref{prime example})
 from Lecture 1,
which was 3,2,2,11,2,7,$p$ with $p \geq 13$, and partial products
1,3,6,12,132,264,1848, $1848p$,  is the sequence
\begin{equation}
\label{perm example}
\xi_1\xi_2\cdots \xi_{20}=10111000011000100000,
\end{equation}

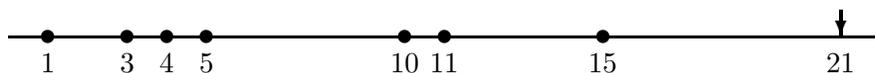
\begin{picture}(500,50)(-20,0)
\thicklines
\put(-15,25){\line(1,0){330}}
\put(0,25){\circle*{5}}
\put(30,25){\circle*{5}}
\put(45,25){\circle*{5}}
\put(60,25){\circle*{5}}

\put(135,25){\circle*{5}}
\put(150,25){\circle*{5}}
\put(210,25){\circle*{5}}
\put(300,35){\vector(0,-1){10}}

\put(0,15){\makebox(0,0){1}}
\put(30,15){\makebox(0,0){3}}
\put(45,15){\makebox(0,0){4}}
\put(60,15){\makebox(0,0){5}}

\put(135,15){\makebox(0,0){10}}
\put(150,15){\makebox(0,0){11}}
\put(210,15){\makebox(0,0){15}}
\put(300,15){\makebox(0,0){21}}
\end{picture}

\noindent
or equivalently,
$$
B_1,\ldots,B_6,B_7 \ = \ 2,1,1,5,1,4,B_7, \  \mbox{ with } B_7 \geq 6.
$$
Do you see the correspondence between
3,2,2,11,2,7,$p$ with $p \geq 13$, and
$10111000011000100000\cdots$\ ?  Writing $p_i$ for $i^{th}$ smallest prime $p_i$,
the  prime and permutations examples match in that the
$k^{th}$ prime on the list  of primes is $p_{B_k}$.

In the Feller coupling, the cycle structure of
a random $\pi \in {\cal S}_n$
has been realized via $C_i(n) =  \# i$-spacings in
$1 \xi_2 \xi_3 \cdots \xi_{n} 1$.  For comparison, consider
$$
C_i(\infty) : =  \# i-\mbox{spacings in }1 \xi_2 \xi_3 \cdots \  = \sum_{k \geq 1} 1(B_k=i).
$$

An easy calculation, (\ref{feller bound}) and
(\ref{feller bound 2}) below,
shows $\p(C_i(n) \neq C_i(\infty))$
$\leq 2/(n+1)$.  Recall that convergence in distribution
for infinite sequences,
as in (\ref{perms limit}),
is equivalent to having convergence for the restriction to the
first $k$ coordinates, for all $k$. Thus
$$
\p( \ (C_1(n),\ldots,C_k(n)) \neq (Z_1,\ldots,Z_k)\  )
\ \leq \ 2k/(n+1) \ \rightarrow 0
$$
for all $k$, so that
$(C_1(n),\ldots,C_k(n)) \Rightarrow (Z_1,\ldots,Z_k),$
and hence \\
$(C_1(n),C_2(n),\ldots) \Rightarrow
(C_1(\infty),C_2(\infty),\ldots)$.
Comparison with (\ref{perms limit}) shows that the
$C_i(\infty)$ are independent, Poisson, with $\e C_i(\infty)=1/i$.

\subsection{A size biased permutation of the multiset having
$i$ with multiplicity $Z_i \sim $ Poisson($1/i)$}\label{Janson}

The above indirect argument, that the $C_i(\infty)$
are independent, Poisson($1/i)$ was presented at
Oberwolfach one morning in August 1993; Erd\H{o}s and Svante Janson were
in the audience. Svante asked if there were a direct proof; I said I
didn't know of one.  Before lunch time, Svante Janson
found and presented the following direct argument.
Start with the $Z_i$, given to be independent, Poisson($1/i$).
Take a random infinite multiset, having $Z_i$ copies of
$s_i := 0^{i-1}1$, the string of length and weight $i$.
Take a size biased permutation of this multiset to get a
list $R_1,R_2,\ldots$, so that by construction, the
number of $i$-spacings in the string $1 R_1 R_2 \cdots$
is $Z_i$, for each $i$.  Calculation (\cite{course}, section 9.1)
shows that the random string
$1 R_1 R_2 \cdots$ of zeros and ones has the same distribution
as $\xi_1 \xi_2 \cdots$\ .  And as an historical note:
there already existed yet another direct argument, by marking Poisson processes.
Jim Pitman describes it this way:
``As observed in Diaconis-Pitman \cite{dp86},
the fact that the numbers of i-spacings in the Bernoulli (1/j) sequence
are independent Poisson (1/i) is an immediate consequence of the
structure of records of a sequence of i.i.d. uniform (0,1) variables
$U_1, U_2, \ldots$\ .
For if $N_1 < N_2 < \cdots$ are the successive record indices, then by
a well known result of R\'enyi the
indicators $1(N_k = j \mbox{ for some }  k)$ are independent Bernoulli $(1/j)$,
and as shown by Ignatov \cite{ignatov81}
the numbers of $i$-spacings in the record sequence are independent
Poisson($1/i$).''

\subsection{Keeping score for the Feller coupling}
\label{score}

Starting from the independent Bernoulli $\xi_i$ with
$\p(\xi_i=1)=1/i$,
and defining $Z_i$ as $Z_i:=C_i(\infty)$, we have a coupling of
the cycle structures $(C_i(n))_i$ for $n=1,2,\ldots$, together
with
the independent Poisson $Z_i$ with $\e Z_i = 1/i$.
Note that for the event that an $i$-spacing occurs with right
end at $k$, the probability is a simple telescoping product:
$$
 \p( \xi_{k-i}\cdots \xi_k = 1\, 0^{i-1}\!1)
= \frac{1}{k-i}\ \frac{k-i}{k-(i-1)} \cdots
\frac{k-3}{k-2} \frac{k-2}{k-1} \ \frac{1}{k}
= \frac{1}{(k-1)k}.
$$
 We can have
$C_i(n)<Z_i$, due to $i$-spacings whose right end occurs after position $n+1$;
the expected number of times this occurs is
\begin{equation}
\label{feller bound}
\sum_{k>n+1} \p( \xi_{k-i}\cdots \xi_k = 1\, 0^{i-1}\!1)
=\sum_{k>n+1} \frac{1}{(k-1)k} = \frac{1}{n+1}
\end{equation}
The only
way that $C_i(n) > Z_i$ can occur is
if the ``artificial'' one at position $n+1$ in $1 \xi_2 \cdots \xi_n 1$
creates an extra $i$-spacing; for each $n$  this can occur for at
most one $i$, and it occurs for each $1 \leq i \leq n$ with the same
probability,
\begin{equation}
\label{feller bound 2}
\p( \xi_{n-i+1} \cdots \xi_n \xi_{n+1}=  1\, 0^{i-1}\!0)
= \frac{1}{n-i+1} \ \frac{n-i+1}{n-i+2} \cdots
\frac{n}{n+1}
 = \frac{1}{n+1}.
\end{equation}

The length $A(n)$ of the first cycle in canonical cycle notation
is precisely the value $i$ for which  this ``extra'' $i$-cycle may
occur,
$$A(n) = n+1 - \max \{j \leq n\!: \xi_j=1\},$$
so that with
\begin{equation}
\label{J for perms}
 J(n) := \max \left\{ \sum_1^l B_l : \  \sum_1^l B_l \leq n-1
\right\}
\end{equation}
we have $J(n)+A(n) = n$.

We now summarize the Feller coupling in a way which matches the coupling
for primes in subsection \ref{sketch}.
Start with independent Poisson random
variables $Z_i$ with $\e Z_i = 1/i$.  Take the
infinite multiset having $Z_i$ copies of $i$.  Take a size biased
permutation $B_1,B_2,\ldots$ of this multiset.  Let $J(n) \in [0,n-1]$
be the
largest partial sum not exceeding $n-1$.  (There is no need to
calculate the distribution of $J(n)$, but it happens to be exactly
uniform on $0,1,\ldots,n-1$, which matches  the harmonic distribution
in (\ref{harmonic}), in the sense that $\log H(n)$ is
approximately uniform on $[0,\log n]$.)  Fill in one extra cycle
length, $A(n) := n-J(n)$; this corresponds to $P_0(n)$.  The
resulting cycle structure, with cycles of lengths
$B_1,B_2,\ldots,B_L$ and $A(n)$
(where $J(n)=B_1+\cdots + B_L$), is the cycle structure of a
random
permutation chosen uniformly from ${\cal S}_n$.

Write ${\bf e}_i:= (0,0,\ldots,0,1,0,\ldots)$
for the unit vector with all zero coordinates except for a one in
position $i$.  The Feller coupling shows that the Wasserstein
insertion-deletion distance for permutations compared to their
independent limit process is at most 2, for all $n$, with a
monotonicity relation as a bonus:
\begin{equation}\label{less than 2}
\e \sum_1^n |C_i(n) - Z_i| \ \leq \frac{2n}{n+1} < 2, \mbox{ and }
\end{equation}
\begin{equation}
\label{monotone Feller}
\mbox{ always } \ \ \ \
(C_1(n),C_2(n),\ldots) \ \leq \ (Z_1,Z_2,\ldots) + {\bf e}_{A(n)}\ .
\end{equation}

The result (\ref{monotone Feller}) for random permutations is
stronger than the analogous result for prime factorizations,
(\ref{JP versus N}), in that there is no exceptionally probability
for an event where monotonicity may fail.  The result for
permutations suggests the following conjecture
 from \cite{AMS1}, for which, in the
style of Erd\H{o}s, I now offer a five hundred dollar prize.

\begin{conjecture} {\bf (\$500 prize offered) }\label{500 conjecture}
For all $n \geq 1$, it is possible to construct
$N(n)$ uniformly distributed from 1 to $n$,
$M(n)$ defined by (\ref{def M}),
and a prime $P(n)$ such that
\begin{equation}
\label{500 question}
\mbox{always } \ \ \  N(n) \ | \ M(n)P(n).
\end{equation}
\end{conjecture}

My reason for believing the conjecture to be true is that
permutations ``fit together perfectly,'' as witnessed by
the Feller coupling, and that primes do so also.  One sense
in which ``primes fit together perfectly'' is that the
weights $\log 2, \log 3, \log 5,\log 7, \log 11, \ldots$
are such that 1)  all multisets of primes have distinct weights,
and 2) the weights of these multisets are evenly spaced:
 $\log 1,\log 2, \log 3, \log 4, \ldots$\ .

A restatement of Conjecture \ref{500 conjecture} in the language
of
stochastic monotonicity, with respect to the partial order of divisors and
multiples,  is that for every $n$, \emph{for some randomized choice} of $P_0(n)$ to
be 1 or a prime factor of $N(n)$, $N(n)/P_0(n)$ lies below $M(n)$
in distribution.  A more specific version of the conjecture, from
\cite{primedw}, is that $P_0(n)$ can be chosen as the first prime
factor of $N(n)$ under a size biased permutation.

In terms of the usual combinatorial language of matchings,
Conjecture  \ref{500 conjecture} may be stated as follows.
For any set $D$ of positive integers, define
$$
l(D) := \{ i: \exists m \in D, \ p \mbox{ prime}, \ \ i \ | \ mp
\},
$$
so that for example $l(\{1\})$ is the set of primes, together
with 1.  Write $[n]$
for $\{1,2,\ldots,n\}$.  The conjecture is that
$\forall n \geq 1, \ \forall D \subset [n],$
$$
\frac{1}{n} \ | \ l(D) \cap [n] \ | \  \geq \ \sum_{m \in D}
\p(M(n)=m) =
\prod_{p \leq n} \left(1-\frac{1}{p}\right) \sum_{m \in D }\frac{1}{m}.
$$

\subsection{At least $2-\epsilon$ indels are needed}

The following extensions to (\ref{less than 2}) aren't obvious,
but help give a full picture of the qualitative behavior
of the Feller coupling. Namely, the positive and negative
parts of the quantity inside the absolute value in the
left side of (\ref{less than 2}) have
$$
 \sum_1^n (C_i(n)-Z_i)^+ \ \rightarrow 1 \ \ \mbox{ in probability
 and expectation, and}
$$
$$
\mbox{for } k=0,1,2,\ldots,
\p\left(  \sum_1^n (Z_i -C_i(n))^+ \ = k\right) \rightarrow p_k >0,
\ \mbox{ with }\sum k p_k =1.
$$
In words, the coupling converts $(C_1(n),\ldots,C_n(n))$ to $(Z_1,\ldots,Z_n)$
with, in the limit, one deletion and a random, mean 1 number of
insertions, using $k$ insertions with probability $p_k$.

A separate argument (see \cite{primedw}) shows that, for {\em any}
coupling, Feller or otherwise, with probability approaching one,
at least one deletion is necessary, i.e.
$1=\lim \p(  \sum_1^n (C_i(n)-Z_i)^+  \ \geq 1)$.  In particular,
$1 \leq \liminf \e  \sum_1^n (C_i(n)-Z_i)^+ $.  Then, since $\e C_i(n)$
$= 1/n$ $= \e Z_i$  for $i=1$ to $n$, the number of insertions
must, on
average, equal the number of deletions.  Hence
$1 \leq \liminf \e  \sum_1^n (Z_i-C_i(n))^+ $, so that
$2 \leq \liminf  \e \sum_1^n |C_i(n) - Z_i|$.  This shows that the Feller
coupling, from the point of view of  (\ref{less than 2}),
is asymptotically optimal.

\section{Lecture 3.  Rescaling space -- to get a scale invariant Poisson process}

\subsection{Review: couplings for primes and permutations}

To review, we have two couplings for growing a system with parameter
$n$.  The coupling for primes grows a random integer $J(n)P_0(n)$ which is
almost uniformly distributed from 1 to $n$.  The coupling for permutations grows a
random permutation which is distributed exactly uniformly in
${\cal S}_n$.

The coupling for primes takes a list $Q_1,Q_2,\ldots$ of primes,
forms their partial products, and pays attention to
$J(n)=Q_1 Q_2\cdots Q_{L(n)}$, the
largest partial product not exceeding $n$.  On a logarithmic
scale, this means that we plot the points $0, \log Q_1$,
$\log Q_1 + \log Q_2, \ldots$, and consider the largest point not
exceeding $\log n$. Think of the values $\log Q_i$ as spacings.
We finish by  filling in one extra spacing, $\log P_0(n)$, to
get to the point $\log( J(n)P_0(n))$ close to, but not to the right of,
$\log n$.  The resulting integer is $J(n)P_0(n) = Q_1 Q_2 \cdots Q_L
P_0$,
with either $L$ or $L+1$ prime factors, depending on whether
$P_0=1$ or not.

The coupling for permutations takes a list $B_1,B_2,\ldots$ of
positive integers, forms their partial sums, and pays attention
to $J(n)$, the largest partial sum not exceeding $n-1$.
We then fill in one extra spacing, of size $A(n)$, to get to the
point $n=J(n)+A(n)$ $=B_1+B_2+\cdots+B_L+A(n)$.  The resulting
cycle structure has $L(n)+1$ cycles.

We review once again, looking only at the sequence of points
plotted, and their spacings.  For primes, the points plotted are
$0$, $\log Q_1$, $\log Q_1 + \log Q_2, \ldots$,  with spacings
$\log Q_1, \log Q_2, \log Q_3, \ldots$.  Conditional on the multiset of
spacings, the order in which they are taken is given by a
size biased permutation.  For each prime $p$,
the number of $k$
such that $\log Q_k=\log p$ is $Z_p$, and the $Z_p$ are independent,
geometrically distributed.
For permutations, the points plotted are
$0$, $B_1$, $B_1 + B_2, \ldots$,  with spacings
$B_1, B_2, B_3, \ldots$.  Conditional on the multiset of
spacings, the order in which they are taken is given by a
size biased permutation.  For each positive integer  $i$,
the number of $k$
such that $B_k=i$ is $Z_i$, and the $Z_i$ are independent,
Poisson($1/i)$.

The underlying reason that primes and permutations have similar
behavior is that for both systems, the {\em spacings}
have the same {\em logarithmic} property:  the expected total
number of spacings of size at most $x$ grows like $\log x$.
For primes, this is the property that
$$
\e \sum_{\log p \leq x} Z_p \  \ = \sum_{p \leq e^x}
\frac{1}{p-1} \ \sim \log x
$$
and for permutations,
$$
\e \sum_{i \leq x} Z_i \ \ = \sum_{i \leq x} \frac{1}{i} \ \sim
\log x.
$$
We pursue this further; even the expression $\gamma + \log x + o(1)$
is common to the two systems --- see (\ref{1/q f}).

\subsection{Limits after scaling space}

What does the sequence $\xi_1 \xi_2 \cdots $ look like
when viewed from a distance?  Encode
$\xi_1 \xi_2 \cdots  \in \{0,1\}^\infty$ as the point process,
i.e random measure,
$$
\sum_{i \geq 1} \xi_i \, \delta(i),
$$
where $\delta(i)$ is the measure placing unit mass at $i$.
Our example was $\xi_1 \xi_2 \cdots = $
$$
10111000011000100000\cdots
\longleftrightarrow \delta(1) + \delta(3) + \delta(4) + \delta(5) +
\delta(10) + \delta(11) + \delta(15) + \cdots \ .
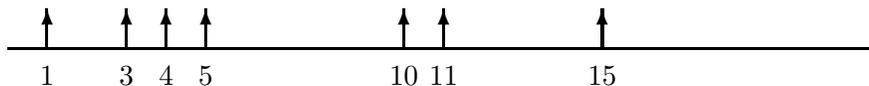
$$
\begin{picture}(500,50)(-20,0)
\thicklines
\put(-15,25){\line(1,0){330}}
\put(0,25){\vector(0,1){15}}
\put(30,25){\vector(0,1){15}}
\put(45,25){\vector(0,1){15}}
\put(60,25){\vector(0,1){15}}

\put(135,25){\vector(0,1){15}}
\put(150,25){\vector(0,1){15}}
\put(210,25){\vector(0,1){15}}

\put(0,15){\makebox(0,0){1}}
\put(30,15){\makebox(0,0){3}}
\put(45,15){\makebox(0,0){4}}
\put(60,15){\makebox(0,0){5}}

\put(135,15){\makebox(0,0){10}}
\put(150,15){\makebox(0,0){11}}
\put(210,15){\makebox(0,0){15}}
\end{picture}

We rescale space:  divide the locations by $x$, to get
$\sum_i \xi_i \, \delta(i/x)$ and take $x \rightarrow
\infty$.
The limit is the ``scale invariant'' Poisson process $\X$
on $(0,\infty)$ with intensity $\frac{1}{x} \ dx$.  The
intensity gives the expected number of points in any interval
$(a,b)$ with $0<a<b<\infty$, which is
$$
\e \X(a,b) = \int_a^b \frac{1}{x} \ dx \ =
\log (b/a).
$$
The Poisson property is that for any disjoint $I_1,I_2,\ldots \subset
(0,\infty)$,
the counts of points in these sets, $\X(I_1), \X(I_2),\ldots$ are
independent, and Poisson distributed.  In particular, for $0<a<b<\infty$
$$
\p( \X(a,b)=0) = \exp(- \e \X(a,b)) = \frac{a}{b}\ .
$$

Likewise, consider the point process $\sum Z_i \, \delta_i$.
This process encodes the multiset of spacings used in the Feller
coupling, without keeping track of the order in which the spacings
are used.
In our example, $\xi_1 \xi_2 \cdots = $
$10111000011000100000\cdots,$
which implies $Z_1 \geq 3, Z_2 \geq 1,Z_4\geq1, Z_5\geq 1$.  We
now
reveal more  about the outcome in this example, and declaring that
$Z_1=3,Z_2=1,Z_3=0,Z_4=1,Z_5=1$, which corresponds to the example
for primes, (\ref{prime example counts}).  And to have a nice picture,
we also declare that $Z_i=0$ for $i=6$ to 13, $Z_{14}=2$,
$Z_{15}=Z_{16}=0,$ $Z_{17}=1$, and $Z_i=0$ for $i=18$ to 25.
Our random measure is

$$
\sum_{i \geq 1} Z_i \, \delta(i) = 3 \delta(1) + \delta(2) +
\delta(4)
+ \delta(5) + \delta(14) + \delta(17) +\cdots \ .
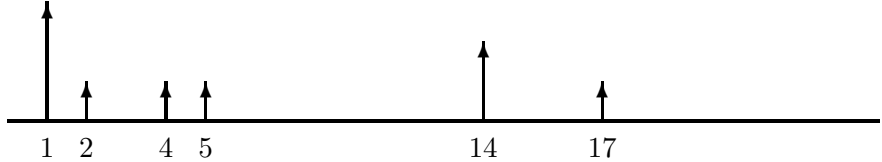
$$
\nopagebreak
\begin{picture}(500,60)(-20,10)
\thicklines
\put(-15,25){\line(1,0){330}}

\put(0,25){\vector(0,1){45}}
\put(15,25){\vector(0,1){15}}
\put(45,25){\vector(0,1){15}}
\put(60,25){\vector(0,1){15}}
\put(165,25){\vector(0,1){30}}
\put(210,25){\vector(0,1){15}}

\put(0,15){\makebox(0,0){1}}
\put(15,15){\makebox(0,0){2}}
\put(45,15){\makebox(0,0){4}}
\put(60,15){\makebox(0,0){5}}
\put(165,15){\makebox(0,0){14}}
\put(210,15){\makebox(0,0){17}}
\end{picture}

This process rescaled, $\sum Z_i \, \delta(i/x)$,
converges in distribution to $\X$:
\begin{equation}
\label{rescale perms}
\sum_{i \geq 1}  Z_i \, \delta(i/x) \ \Rightarrow \X \ \ \mbox{
as } x \rightarrow \infty.
\end{equation}
  Recall that $Z_i$ is the
number of $i$-spacings
of $ \xi_1 \xi_2 \cdots$\ .  We have a
process $\xi$ whose rescaled limit is $\X$, and the spacings of
this process $\xi$ also have rescaled limit $\X$.  Hence it
is plausible to guess that $\X$ is equal in distribution to its
own process of spacings.

\subsection{The scale invariant spacing lemma}\label{sect scale invar}
\label{sect exp(-wy)}

Indeed, the process $\X$ {\em is equal in distribution} to its own
spacings.  Since $\X$ has no multiple points (unlike
$\sum Z_i \delta(i),$ which can have multiple points), $\X$ \
can be encoded as a random set
 $\X = \{ X_i\!\! : i \in \BZ \}$ $ \subset (0,\infty)$, with the points
indexed so that for all $i \in \BZ$, $X_i < X_{i+1}$.  With
such indexing, the spacings are the points
$$
Y_i := X_{i+1}-X_i,
$$
and the statement that the spacings of $\X$ have the same
distribution as $\X$ is
$$
\{ Y_i\!: i \in \BZ \} \indist \X,
 \ \ \mbox{ and }1 = \p( Y_i \neq Y_j \ \forall i \neq j \in \BZ).
$$

The proof of this, from \cite{primedw},
is inspired by Janson's proof in
section \ref{Janson} that if the multiset with $Z_i$
copies of $i$, to be used as spacings,
 is placed in a size biased permutation, then the
set of points is encoded by the Bernoulli ($1/i$)
sequence $1 \xi_2 \xi_3 \cdots$\ .  Viewing this in terms of the
limits under spatial rescaling, it suggests that if the points of
$\X$, to be used as spacings, are placed in a size biased
permutation, then the set of partial sums also is distributed
as the scale invariant Poisson process.

The size biased permutation of $\X$ has
a very pleasing, symmetric expression.
Recall from section \ref{size bias section}
that a size biased permutation can be generated by using,
for an object of weight $x$, an exponentially distributed
label $W$, with density $f_W(w) = x e^{-wx}$.  To emphasize
that the weights are also random, and we have conditioned on
the values of these weights, we write this with
the notation for a conditional density given $x$, that is
$f_{W|X}(w|x) \ dw = x e^{-wx} \ dw$.
Starting with the scale invariant Poisson process
$\X = \{ X_i \}$, with intensity $f_X(x) \ dx = (1/x) \ dx$ on
$(0,\infty)$, and attaching label $W_i$ to $X_i$, the process
$\{ (W_i,X_i): i \in \BZ\}$ of points with labels is a Poisson
process  with intensity $f_X(x) \ dx \ f_{W|X}(w|x) \ dw\ =  $
$$
f_{W,X}(w,x) \ dw \ dx  = \
e^{-wx} \ dw \ dx , \ \ \ \ (w,x) \in (0,\infty)^2.
$$
Note the beautiful and  perfect symmetry between the
points  and  labels: the distribution of the
process is invariant
under $(w,x) \mapsto (x,w)$.
No such symmetry is  possible for
our other size biased permutations, since the points have discrete
support, such as $\BN$ or $\{ \log p: \ p $ is prime~$\}$, while
the labels are continuously distributed in $(0,\infty)$.
The distribution of the
process is also invariant
under rescaling of the form
$(w,x) \mapsto (cw,x/c)$; we apply this with $c=2$ in the picture
below, and with $c=\log n$ in the proof of Lemma
\ref{density of J lemma}.

The picture shows the Poisson process with intensity $e^{-wx}$
restricted to the region $(0,b]^2$ with $b=5$. The number of points
in this region
is Poisson with mean $\int_0^b (1-e^{-bx}) \ dx$; this mean is 3.796
for $b=5$, and would be $8.401$ for $b=50$.

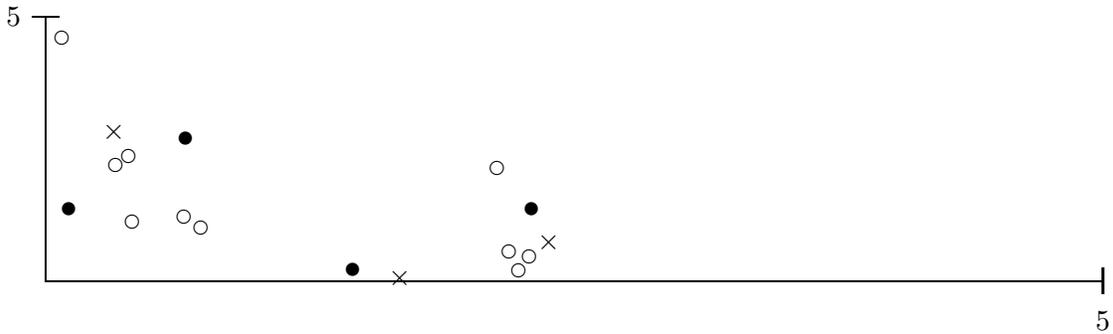
\begin{picture}(400,150)(35,-25)
\put(0,0){\line(1,0){400}}
\put(0,0){\line(0,1){100}}
\put(-5,100){\line(1,0){10}}
\put(400,-5){\line(0,1){10}}
\put(-12,100){\makebox(0,0){5}}
\put(400,-15){\makebox(0,0){5}}

\put(53.04,54.14){\circle*{5}}
\put(8.846,27.42){\circle*{5}}
\put(183.797,27.288){\circle*{5}}
\put(116.31,4.451){\circle*{5}}

\put(6.166,92.158){\circle{5}}
\put(31.46,47.249){\circle{5}}
\put(26.53,43.995){\circle{5}}
\put(170.85,42.64){\circle{5}}
\put(52.35,24.497){\circle{5}}
\put(32.84,22.41){\circle{5}}
\put(58.8,20.36){\circle{5}}
\put(175.267,11.076){\circle{5}}
\put(182.97,9.252){\circle{5}}
\put(178.859,4.1){\circle{5}}
\put(25.84,56.55){\makebox(0,0){$\times$}}
\put(190.39,14.63){\makebox(0,0){$\times$}}
\put(134.03,1.003){\makebox(0,0){$\times$}}

\end{picture}
 We show three
realizations; the first experiment, shown with solid circles, had  4
points in this region,  the second experiment, shown with open circles
had 10 points, and the third experiment, shown with ``$\times$'' had 3 points.
This is honest simulation; I did only three runs, even though they
hardly look typical to my naive eye.  Try to visualize each of
the three runs by itself.

The proof of the scale invariant spacing lemma goes as follows.
Start with the point process on $(0,\infty)^2$ with intensity
$e^{-wy} \ dw \ dy$, and index the points as $(W_i,Y_i)$ with
$W_i > W_{i+1}$ for all $i \in \BZ$. (The labels
$W_i$ are all distinct, with probability one, and we remove
 from the probability space the complementary
event.)
Notice the reverse direction for
the deterministic inequality;
small indices $i \longleftrightarrow$  large labels
$W_i$,
which
tend to go with small weights $Y_i$.
This gives us a set of
points $\{ Y_i \}$ having the distribution of the scale invariant
Poisson process, indexed in order of a size biased permutation, tending
 from small to large.
Define the points $X_j$ to be the partial sums of the $Y_i$ in
this order:
$$
\mbox{ for } i \in \BZ, \ \ \ X_j :=  \sum_{-\infty < i < j}
Y_i
$$
This gives a set of points $0 < \cdots < X_i < X_{i+1} < \cdots < \infty$,
and further calculation shows that the distribution of
$\{X_i: \ i \in \BZ\}$
is that of a scale invariant Poisson process
with intensity $(1/x) \ dx$.
By construction, the spacings of the $X_i$ are the points $Y_i$,
with the desired distribution.

\subsection{Primes and the scale invariant Poisson}

How does the process of  points in the coupling for primes appear,
viewed from a distance?  Recall, the points are
$0, \log Q_1, \log Q_1 +\log Q_2, \cdots$,
and their spacings, $\log Q_1, \log Q_2,\ldots$
are such that $\log p$  occurs $Z_p$ times, where the $Z_p$
are independent and geometrically distributed.  A process which
encodes the multiset of spacings, in a form suitable for spatial
rescaling, is the random measure
\begin{equation}
\label{primes measure}
\sum_{i \geq 1} \delta(\log Q_i)
\ = \ \sum_p Z_p \ \delta(\log p).
\end{equation}
The direct analog of (\ref{rescale perms}) would be the
statement that
\begin{equation}
\label{rescale primes}
\sum_{p}  Z_p \, \delta(\log p /x)
 \ \Rightarrow \X \ \ \mbox{ as } x \rightarrow \infty.
\end{equation}
Standard probability theory reduces this to showing that for
$0<a<b<\infty$, the expected mass that the rescaled measure gives
to $(a,b)$ converges to $\log(b/a)$ as $x\rightarrow \infty$.
The mass is $\sum_{ax < \log p <bx} \e Z_p$ $=\sum_{e^{ax} < p < e^{bx}}
1/(p-1)$,
so that sufficient number-theoretic knowledge needed to prove
(\ref{rescale primes}) is that as $y \rightarrow \infty$,
$\sum_{p \leq y} 1/p $ $=B+ \log \log y + o(1)$,
for some constant $B$,  a statement weaker than the prime
number theorem.

But the random measure $\sum_{p}  Z_p \, \delta(\log p )$
=$\sum_{i \geq 1} \delta(\log Q_i)$
is much closer to $\X$ than the rescaling relation
(\ref{rescale primes}) shows. Namely, we can match up
points so that the total amount of displacement needed to convert one
sequence to the other is finite (in expectation, and therefore
almost surely.)  This coupling comes from \cite{primedw}, where it
is combined with the total variation distance
approximation of the small prime factors of a
uniform integer, to give a metrized version of the Poisson
process approximation for the ``intermediate'' prime divisors,
 from De Koninck and Galambos \cite{galambos}.
The points $X_i$ of $\X$ are indexed by $i \in \BZ$, while the
points $Q_i$ are indexed by $i \in \BN$, so we define $Q_i = 1$
for $i \leq 0$.  The claim is that we can construct the $Q_i$ and $X_i$
on a single probability space, so that
\begin{equation}\label{couple geometric}
  \e \sum_{i \in \BZ} | X_i - \log Q_i | < \infty\ .
\end{equation}

\subsubsection{\bf Ignoring the difference between geometric and
Poisson}

In (\ref{couple geometric}) the chief obstacle is that $\X$ is
intrinsically a Poisson process, while the prescribed multiplicity
of $p$ in the sequence $Q_i$ is not Poisson, but rather a
geometric, $Z_p$.
We make our task easier if we change the prescription to the following:
for every
$p$, $A_p= \sum_i 1(Q_i=p)$, where
the $A_p$ are independent, Poisson($1/p$).  The two versions of
the prescribed counts, $Z_p$ and
$A_p$ are close, and can be coupled with $\e |Z_p-A_p| =
1/(p(p-1))$.
To convert the coupling with Poisson multiplicities into a
coupling with geometric multiplicities, we can move
$|Z_p-A_p|$ copies of $\log p$, each at most through distance $\log
p$, because there is an infinite supply of
points $\log Q_j = 0$ to swap with.
Since
\begin{equation}\label{perturb}
\sum \log p /(p(p-1)) < \infty,
\end{equation}
this perturbation
is
absorbed by the right side of (\ref{couple geometric}).

A simple way to handle the scale invariant Poisson process
on $(0,\infty)$ is to
start with the translation invariant Poisson $\L$ with intensity
$ 1 \ dx$ on $\BR$ --- a process for which the number of points in
an interval of length $x$ is Poisson distributed, with mean $x$.
If the points of this process are $L_i$ for $i \in \BZ$, then
setting
\begin{equation}\label{exponential map}
  X_i := e^{L_i}
\end{equation}
gives the points of the scale invariant Poisson process $\X$.
To get $A_p$ to be Poisson with mean $1/p$, all we need to do is
assign some interval of length $1/p$, and let $A_p$ count how many
points of $\L$ land in that interval.  Using disjoint intervals
for different $p$ makes the $A_p$ mutually independent.

The error estimate for the prime number theorem (see e.g.
\cite{rs}, or
\cite{tenenbaum} section 4.1) implies the well known estimate
\begin{equation}\label{1/p rate}
  \sum_{p \leq x} \frac{1}{p} = B + \log \log x \ +
  O \left(\exp(-c  \sqrt{\log x} \, )\right)
\end{equation}
as $x \rightarrow \infty$, for some $c>0$, with  constant
$$
B :=  \gamma - \sum_{k \geq 2} \sum_p \frac{1}{kp^k}
\ \doteq
.261497,
$$
where $\gamma$ is  Euler's constant,
$\gamma \doteq .5772$\ .
We use this in the form
\begin{equation}\label{1/p f}
  f(x) := -B + \sum_{p \leq e^x} \frac{1}{p} \ \ \
   =\log x + O \left(\exp(-c  \sqrt{ x} \, )\right).
\end{equation}
The best upper bound for (\ref{1/p rate})
is $O( \exp(-c(\log x)^{3/5}(\ll  x)^{-1/5}))$, but
more easily stated estimate
leads to the same order error bounds in our work.

We define a function $g: \BR \rightarrow \{0,\log 2,\log 3,\ldots \}$
to be essentially the inverse of $f$.
Specifically,  $g$ has $(-\infty,-B] \mapsto 0$, $(-B,-B+1/2] \mapsto
\log 2,$ $(-B+1/2,-B+1/2+1/3] \mapsto \log 3, \ldots$\ \, .  Since $f$
is close to the log function, $g$ is close to the exponential
function.  Now let $h: (0,\infty)
\rightarrow \{0,\log 2,\log 3,\ldots \}$ be defined by
$h(x) := g( \log x)$, so that $h$ is close to the identity
function on $(0,\infty)$, and can be applied directly to the
points $X_i$ of the scale invariant Poisson process on $(0,\infty)$.
We have, under $h$, $(0,e^{-B}] \mapsto 0$, $(e^{-B},e^{-B+1/2}] \mapsto
\log 2,$ $(e^{-B+1/2},e^{-B+1/2+1/3}] \mapsto \log 3, \ldots$\ \, .
This map $h$ is such that  the number of $X_i$ with
$h(X_i)=\log p$ is Poisson(1/p), independently for all primes $p$,
and $h$ is so close to the
identity function that $\e \sum_{i \in \BZ} |X_i - h(X_i)| <
\infty$.

The illustration below show the ``bins'' for the translation
invariant Poisson process on $\BR$, with all arrivals to the left
of $-B$ representing ones, arrivals between $-B$ and $-B+1/2$
representing twos, and so on.  The length of the bin corresponding
to $p$ is $1/p$.  We have faked arrivals to
correspond to the example (\ref{prime example counts})
 from Lecture 1, with $Z_2=3,Z_3=1,Z_5=0,Z_7=1,Z_{11}=1$.
The reader should \emph{imagine} an arrow labeled
``exponential map,'' pointing to a second picture, labeled
``Bins for Poisson $(1/x \ dx$ on
$(0,\infty$))'', with dividing markers at
$e^{-B},e^{-B+1/2},e^{-B+1/2+1/3}$, and so on.

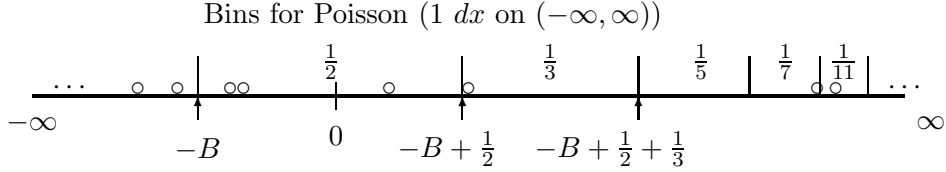
\begin{picture}(500,70)(-20,0)
\put(50,50){\makebox(0,0)[bl]{Bins for Poisson ($1 \ dx$
on $(-\infty,\infty)$)}}
\thicklines
\put(-15,25){\line(1,0){330}}
\thinlines

\put(47.7,25){\line(0,1){15}}      
\put(147.7,25){\line(0,1){15}}    
\put(214.367,25){\line(0,1){15}}
\put(256.367,25){\line(0,1){15}}
\put(282.938,25){\line(0,1){15}}
\put(301.12,25){\line(0,1){15}}

\put(47.7,16){\vector(0,1){9}}      
\put(147.7,16){\vector(0,1){9}}    
\put(214.367,16){\vector(0,1){9}}

\put(100,20){\line(0,1){10}} 

\put(-15,15){\makebox(0,0){$-\infty$}}\put(325,15){\makebox(0,0){$\infty$}}
\put(47.7,5){\makebox(0,0){$-B$}}      
\put(100,10){\makebox(0,0){0}}
\put(142,5){\makebox(0,0){$-B+\frac{1}{2}$}}    
\put(204,5){\makebox(0,0){$-B+\frac{1}{2}+\frac{1}{3}$}}

\put(98,38){\makebox(0,0){$\frac{1}{2}$}}
\put(181,38){\makebox(0,0){$\frac{1}{3}$}}
\put(238,38){\makebox(0,0){$\frac{1}{5}$}}
\put(270,38){\makebox(0,0){$\frac{1}{7}$}}
\put(292,38){\makebox(0,0){$\frac{1}{11}$}}
\put(40,28){\circle{4}}\put(25,28){\circle{4}}\put(0,28){\makebox(0,0){$\cdots$}}
\put(60,28){\circle{4}}\put(65,28){\circle{4}}\put(120,28){\circle{4}}  
\put(150,28){\circle{4}} 
\put(282,28){\circle{4}}    
\put(289,28){\circle{4}}   
\put(316,28){\makebox(0,0){$\cdots$}}
\end{picture}

\begin{theorem}
\label{Poisson 1/p thm}
Let $\X$ be the scale invariant Poisson process on $(0,\infty)$,
with intensity $(1/x) \ dx$, and points $X_i, i \in \BZ$.  Define
$Q_i$ by $\log Q_i = h(X_i)$.  Then every $Q_i$ is either
one or a prime, and for every prime $p$, $A_p := \sum_i 1(Q_i=p)$
is Poisson distributed, with $\e A_p = 1/p$.  The $A_p$ are
mutually independent, and with $f$ defined by (\ref{1/p f}),
\begin{equation}\label{Poisson 1/p distance}
  \e \sum_{i \in \BZ} | X_i - \log Q_i |  \
  = \int_0^\infty |f(x) - \log x | \ dx  \ < \infty.
\end{equation}
\end{theorem}
\begin{proof}
Constructing $\X$ from $\L$ as in the discussion above, so that
$X_i = \exp(L_i)$, we have $\log Q_i = h(X_i) = g(L_i)$.
This construction makes it obvious that the $A_p$ are independent, Poisson, with
$\e A_p = 1/p$.  To prove (\ref{Poisson 1/p distance}), we argue that
\begin{equation}\label{area bound}
\e \sum_{i \in \BZ} |\log Q_i -X_i |
= \e \sum |g(L_i)-\exp(L_i)| = \int_{-\infty}^\infty |g(l)-e^l| \ dl =: a_0,
\end{equation}
where we have used Fubini to express an expectation of a sum over
points of $\L$ in terms of the intensity of $\L$. Now
\begin{equation}\label{a0}
a_0 = \int_{-\infty}^\infty |g(l)-e^l| \ dl = \int_0^\infty |f(x)-\log x| \ dx <
\infty,
\end{equation}
where the equality holds because $a_0$
may be interpreted as the area between the
graphs of $g$ and the exponential function, and equals
the area between the graphs of $f$ and the logarithm (with vertical segments
added to span the jumps of $f$ and of $g$.)
The convergence of this last integral at infinity follows from the bound
(\ref{1/p f}), and the convergence of the integral
at 0  follows
simply as $\int_0^{\log 2} |f(x)-\log x| \ dx $
$= \int_0^{\log 2} (-B - \log x) \ dx
< \infty$.
\end{proof}

\subsubsection{\bf Exploiting the relation between geometric and
Poisson}\label{sect geometric to Poisson}

The argument at (\ref{perturb}), for changing the coupling with
$A_p \sim $ Poisson($1/p$) copies of $p$ to one with $Z_p \sim$
Geometric, is not pretty, but more importantly, it is a
non-explicit procedure and makes it hard to control the size
biased
permutation of the $\log Q_i$ and their partial products.
Carrying out the following explicit coupling turned out to be the key
to being able to analyze the coupling for primes described in
Lecture 1.  A simple observation, that the geometric distribution
is compound Poisson, makes everything work!  Every baby {\em
should know} the following standard lemma, which we will apply with $a=1/p$
and corresponding  random variables $Z_p$ and $\Ak_p \equiv A_{p^k}$.
The added entropy involved in the integer partition occurring in
(\ref{compound Poisson general}) is discussed later, at
(\ref{entropy increment}).

\def\Ak{A^{(k)}}
\begin{lemma}
Let $Z$ be geometric with parameter $a \in (0,1)$,
i.e.~$\p(Z \geq k) =a^k$.   Let $A \equiv A^{(1)}$
be Poisson with $\e A=a$, and more generally,
for $k=1,2,3,\ldots$, let $A^{(k)}$ be Poisson with
$\e \Ak = a^k/k$, with the $\Ak$ independent.
Then
\begin{equation}\label{compound Poisson general}
 Z \indist \sum_{k\geq 1} k \Ak \ = A + \sum_{k\geq 2} k \Ak \ .
\end{equation}
\end{lemma}
\begin{proof}
Recall that the probability generating
functions of geometric and Poisson are given by
$$
\e s^Z=\sum_{j\geq 0} s^j (1-a) a^j = \frac{1-a}{1-as},
\ \ \e s^A=\sum_{j \geq 0} s^j e^{-a} a^j/j! = e^{a(s-1)},
$$
with $|s|<1$,
so that $\e s^{kA}=e^{a(s^k-1)}$ and $\e s^{k \Ak} = e^{((as)^k-a^k)/k}.$
Writing
$$
\log \frac{1-a}{1-as} = \sum_{k \geq 1} \frac{(as)^k -a^k}{k}
$$
proves (\ref{compound Poisson general}).
\end{proof}

Notice also that the last expression in
(\ref{compound Poisson general})
demonstrates stochastic domination $Z_p \geq_d
A_p$, so that in the expressions $|Z_p - A_p|$ in the argument
before (\ref{perturb}), the absolute value signs weren't needed!

For us,
$$
Z_p = \sum_{k \geq 1} k A_{p^k}
$$
with $p$ prime, $k\geq 1$, and
$$
q=p^k, \ \ A_q \sim \mbox{ Poisson}, \ \ \e A_q = \frac{1}{kq}.
$$
The coupling for primes described in Lecture 1 had a multiset with
$Z_p$ copies of $p$, taken in order of a size biased permutation.
A far more tractable coupling has a multiset of primes and
prime powers, with $A_q$ copies of the prime power $q=p^k$, and
the objects $Q_i$  (re-using the same notation $Q_i$ as before)
to be taken in size biased permutation are these prime
powers.  In the remainder of this, Section
\ref{sect geometric to Poisson}, we will study the coupling of
this
multiset of prime powers to the scale invariant Poisson.  Then in
Section \ref{coupling 2} we will take a size biased permutation,
to give our second coupling for growing a random integer.  We will
analyze this second coupling in detail; the first coupling for growing
a random integer, from Lecture 1,  can be
analyzed as a perturbation of the second coupling --- but we
won't present the details of the comparison, which are similar in
spirit to proofs of lemmas \ref{task 1} and \ref{task 2}.

The modified version of (\ref{1/p rate}) is
\begin{equation}\label{1/pk rate}
  \sum_{q = p^k \leq x} \frac{1}{kq} = \gamma + \log \log x \ +
  O \left(\exp(-c  \sqrt{\log x} \, )\right)
\end{equation}
and our modified version of (\ref{1/p f}) is
\begin{equation}\label{1/q f}
  f(x) := -\gamma + \sum_{q=p^k \leq e^x} \frac{1}{qk} \ \ \
   =\log x + O \left(\exp(-c  \sqrt{ x} \, )\right).
\end{equation}
Likewise, we modify $g$ to be essentially the  inverse of this $f$,
with $g(t)=0$ if $t \in (-\infty,-\gamma],$ and  if $q>1$ is a
prime power, $l= \log q$ and $\log x \in (f(l-),f(l)]$ then
$g(x)=l$.  We define $h = g \circ \log$, so that
our modified version of the map $h$, which can be applied to
the
points of the scale invariant Poisson process, has
$$
(0,e^{-\gamma}] \mapsto 0, \ \ (e^{-\gamma},e^{-\gamma+1/2}] \mapsto \log 2,
\ \ (e^{-\gamma+1/2},e^{-\gamma+1/2+1/3}] \mapsto \log 3,
$$
\begin{equation}\label{modified h}
(\exp(-\gamma+1/2+1/3),\exp(-\gamma+1/2+1/3+1/8)] \mapsto \log 4,
\end{equation}
and so on.  Notice that in the last endpoint given above we have
 $1/8 = 1/(kq)$
for $q=p^k=4$.  To summarize formally, $h$ is given by the recipe:  $h(x)=0$
if $0<x\leq e^{-\gamma}$, and if $q>1$ is a
prime power, $l= \log q$ and $\log x \in (f(l-),f(l)]$, then
$h(x)=l$, with $f$ as in (\ref{1/q f}).

\begin{theorem}
\label{Poisson 1/q thm}
Let $\X$ be the scale invariant Poisson process on $(0,\infty)$,
with intensity $(1/x) \ dx$, and points $X_i, i \in \BZ$.  Define
points $Q_i$ by $\log Q_i = h(X_i)$, for the function $h$ described at
(\ref{modified h}).
Then every $Q_i$ is either
one or a prime power, and for every  $q=p^k$, with $p$ prime
and $k \geq 1$,  $A_q := \sum_i 1(Q_i=q)$
is Poisson distributed, with $\e A_q = 1/(kq)$.  The $A_q$ are
mutually independent, and with $f$ defined by (\ref{1/q f}),
\begin{equation}\label{Poisson 1/q distance}
  \e \sum_{i \in \BZ} | X_i - \log Q_i |  \
  = \int_0^\infty |f(x) - \log x | \ dx  \ < \infty.
\end{equation}
\end{theorem}
\begin{proof}
Constructing $\X$ from $\L$ as in the discussion above, so that
$X_i = \exp(L_i)$, we have $\log Q_i = h(X_i) = g(L_i)$.
This construction makes it obvious that the $A_q$ are independent, Poisson, with
$\e A_{p^k} = 1/(kq)$.  To prove (\ref{Poisson 1/q distance}), we argue that
\begin{equation}\label{area bound q}
\e \sum_{i \in \BZ} |\log Q_i -X_i |
= \e \sum |g(L_i)-\exp(L_i)| = \int_{-\infty}^\infty |g(l)-e^l| \ dl =: b_0,
\end{equation}
where we have used Fubini to express an expectation of a sum over
points of $\L$ in terms of the intensity of $\L$. Now
\begin{equation}\label{b0}
b_0 = \int_{-\infty}^\infty |g(l)-e^l| \ dl = \int_0^\infty |f(x)-\log x| \ dx <
\infty,
\end{equation}
where the equality holds because $b_0$
may be interpreted as the area between the
graphs of $g$ and the exponential function, and equals
the area between the graphs of $f$ and the logarithm.
The convergence of this last integral at infinity follows from the bound
(\ref{1/q f}), and the convergence of the integral
at 0  follows
simply as $\int_0^{\log 2} |f(x)-\log x| \ dx $
$= \int_0^{\log 2} (-\gamma - \log x) \ dx
< \infty$.
\end{proof}

\subsection{The size biased permutation of the multiset having
$\log p^k$ with multiplicity
$A_{p^k} \sim$ Poisson$(1/(kp^k))$}
\label{coupling 2}

Just as the multiset with $Z_p$ copies of each prime $p$, for
independent $Z_p \sim$
geometric($1/p)$, is the ``natural random infinite
multiset of primes,'' the multiset with $A_q$ copies of each prime
power $q=p^k>1$, for independent $A_q \sim $ Poisson($1/(kq))$,
is the \emph{ natural random infinite multiset of prime powers}.
The latter multiset may be viewed as the former, with auxiliary
randomization, picking a partition of the integer $Z_p$,
independently for
each $p$.

[Incidentally, the amount of additional
information in our partitioning of the $Z_p$
is small; it is approximately .612433379 bits,
computed as follows.
Recall that for
a discrete random variable $X$ with $\p(X=X_i)=p_i>0$,
$\sum p_i =1$, the \emph{entropy} is $h(X) := -\sum p_i \log p_i$.
For the geometrically distributed $Z$ in
(\ref{compound Poisson general}),
$h(Z)=-\log(1-a)-a/(1-a)\, \log a$
$=\sum_{k \geq 1} (a^k/k)(1+\log(1/a^k))$,
and for $A \sim$ Poisson($x$) we have
$h(A)=x+x \log(1/x) + \sum_{j \geq 2} \p(A \geq j)\log j$,
which we apply with $x=a^k/k$ for $k=1,2,3,\ldots \, .$
Recall that the entropy of an independent process, such as
$(A^{(1)},A^{(2)},\ldots)$, is the sum of the entropies of the
coordinates.
Thus in (\ref{compound Poisson general}),
the additional information needed to partition $Z$ into
$\sum_{k \geq 1} k \Ak$ is $d(a):=\sum_{k \geq 1} h(\Ak)\ - h(Z)=$
\begin{equation}
\label{entropy increment}
 \sum_{k \geq 1} \left(
\frac{a^k}{k}
+ \frac{a^k}{k} \log\frac{k}{a^k}
+ \sum_{j \geq 2} \p(\Ak \geq j)\log j
\right)
-
 \sum_{k \geq 1} \frac{a^k}{k}(1 + \log \frac{1}{a^k})
\end{equation}
$=a^2 \log 2 +O(a^3 \log a)$ as $a \rightarrow 0+$.
Numerical evaluation, with logs taken base 2,  gives
$d(1/2) \doteq $.375076
as the additional information needed to partition
$Z_2$, approximately .13879 for $Z_3$, and approximately
.612433379 for the
sum over all primes.]

In Lecture 1 we described a procedure for ``growing'' a random
integer, based on a size biased permutation of the multiset having
each prime $p$, taken as an object of weight $\log p$,  with
multiplicity $Z_p \sim $ Geometric ($1/p)$.  While relatively simple
to
describe, this coupling is rather hard to analyze directly.  The
reason that it is hard to analyze directly is that the size biased
permutation involves using  exponentially distributed
labels $W_i$, and the resulting two-dimensional process, with points of
the form $(W_i,\log Q_i),$ does not have a simple structure.  By
changing the objects to be prime powers $q=p^k$, with weight
$\log q$
and multiplicity $A_q \sim $ Poisson $1/(kq)$ the total
number of factors of $p$ is still distributed as $Z_p \sim $
Geometric $(1/p)$, because $\sum_{k \geq 1} k A_{p^k} \indist
Z_p$,
but now the size biased permutation is tractable.

The size biased permutation is tractable  because
attaching conditionally independent labels $W_i$
to the Poisson process with points $\log Q_i$ yields a
two-dimensional Poisson process with points $(W_i,\log Q_i)$  --
this is an instance of the ``labeling'' theorem; see for example
\cite{kingman}.

This process $\{ (W_i,\log Q_i): \ i \in \BN \}$
is very close to the Poisson process of Section
\ref{sect exp(-wy)} on $(0,\infty)^2$ with intensity
$e^{-wy} \ dw \ dy $, which was used to give a size biased
permutation of the scale invariant Poisson process.  However,
our two-dimensional process now is neither discrete nor
continuous.  It is   supported on a
one dimensional subset of the positive quadrant, formed
by the half lines
$w>0, y=\log q$ for some prime power $q=p^k>1$.  Its intensity
on the line $y =\log q$
is the product of the discrete intensity, $f_Y(y) = 1 / (kq)$ for the weight
$y$, times the continuous density for exponentially distributed
label $w$, $f_{W|Y}(w|y) \ dw = y e^{-wy} \ dw$.  That is,
for $y = \log q, \, q = p^k,$
\begin{equation}\label{2d intensity for primes}
f_{W,Y}(w,y) \ dw =f_Y(y)\  f_{W|Y}(w|y) \ dw
= \frac{\log q}{kq} \  e^{-wy} \ dw.
\end{equation}

The next lemma gives an exact expression for the density of
$J(n)$, and will let us see in Lemma \ref{lemma j bound}
that the distribution of
$J(n)$ is close to the harmonic distribution (\ref{harmonic})
on $1,2,\ldots,n$:
in (\ref{simpler 1}),
$d(\log \log n^\beta)$ is close to $d\beta / \beta$,
and the zeta function has a simple pole at one,
so the expression in (\ref{simpler 1}) is close
to
\begin{equation}\label{compare harmonic}
 \frac{1}{i} \ \int_{\beta > 1-\alpha} \frac{d \beta}{\beta} \
\int_{c>0} \beta e^{-\beta c}  e^{-\alpha c} \frac{c}{\log n} \ dc
\ = \ \ \frac{1}{i \log n}.
\end{equation}

\begin{lemma}\label{density of J lemma}
\label{density of J}
For the $Q_i$ taken in order of decreasing labels $W_i$, let
$J(n)$ be the largest partial product
with $J(n) \leq n$.  Then for $i=n^\alpha =1$ to $n$,
with $q=p^k=n^\beta$ ranging over prime powers,
\begin{equation}\label{decompose}
\p(J(n)=i) = \sum_{q: \ \beta > 1-\alpha} \frac{1}{kq} \
  \int_{c>0} \beta  \ e^{-\beta c  } \  \
\frac{i^{-1-c/\log n}}{\zeta(1+c/\log n)} \ dc
\end{equation}
\begin{equation}\label{simpler 1}
= \frac{1}{i} \  \sum_{q: \ \beta > 1-\alpha} \frac{1}{kq} \
  \int_{c>0} \beta  \ e^{-\beta c  } \  \
\frac{e^{-\alpha c}}{\zeta(1+c/\log n)} \ dc.
\end{equation}
\end{lemma}

\begin{proof}
Write the labels $W_i$ as $W_i=S_i/\log Q_i$ as in
Section \ref{size bias section}, with $S_1,S_2,\ldots$
being  iid standard exponentials, independent of $Q_1,Q_2,\ldots \ .$
For any prime power $q > 1$ and
constant $t \geq 0$, the number $
A_q(t) := \sum_i \bone (Q_i=q,S_i/\log q >t)$
of occurrences of $q$ with label greater than $t$ is
Poisson distributed with $\e A_q(t) = \e A_q \ \p(S_i/\log q >t)$
$= 1/(kq) \ q^{-t}$. The  $A_q(t)$
jointly for all $q$ are
independent.  For each prime $p$ let
\begin{equation}\label{def Zpt}
Z_p(t) := \sum_{k \geq 1} k A_{p^k}(t).
\end{equation}
Note that by (\ref{compound Poisson general})
the distribution of $Z_p(t)$ is geometric
with $\p (Z_p(t) \geq~j)$   $ = (1/p^{1+t})^j$ and
\begin{equation}\label{dist Zpt}
\p (Z_p(t)=j) = (1 - p^{-1-t}) \ \left( \frac{1}{p^{1+t}} \right)^j.
\end{equation}

For any $t>0$, consider the product
$I_t$ of all primes and prime powers
having label strictly greater than $t$, i.e.
$$
I_t :=  \prod_{q=p^k} q^{A_q(t)} = \prod p^{Z_p(t)}.
$$
Using (\ref{dist Zpt}) and
the independence of the $Z_p(t)$ over all primes $p$,
 for any $i \geq 1$ we have
\begin{equation}\label{dist It}
\p (I_t = i) = \prod (1-p^{-1-t}) \ \ i^{-1-t} \ \ = \frac{i^{-1-t}}{\zeta(1+t)}.
\end{equation}

With probability one, all labels $W_i$ are distinct.
For any $t>0$ there are, with probability one, only
finitely many labels greater than $t$  --- this follows from
(\ref{dist It}), which summed over $i=1,2,\ldots$ yields 1.
There are infinitely many labels, with probability one,
because the total intensity of our Poisson process is
infinite --- it is $\sum_{q=p^k} \e A_q = \sum_{q=p^k} 1/(kq)$
$ > \sum_p (1/p) = \infty$.   For these three reason combined,
with probability one, as $t$ decreases from infinity to zero, the partial
products $I_t$ increase from 1 to infinity, and each increase corresponds
to factoring in one new factor $Q_i$, taken in decreasing order of
their labels $L_i= S_i/ \log Q_i$.

 Let $Q^*(n)$  be the new factor that first
takes the partial product beyond $n$, and let  $T(n)$ be its label.
Our product $J(n)$ will equal $I_t$ for $t=T(n)$.
We consider the joint distribution of $(Q^*(n),T(n),J(n))$.
Write $i = n^\alpha$ for the test value for $J(n)$,
$q=p^k=n^\beta$ for the test value for  $Q^*(n)$, and
$c$ for  $ T(n)  \log n$, which, as $\log n$ times the
label for $q$,  is
exponentially distributed with rate $\log q/ \log n = \beta$.
Summing over $q$ and integrating over $c$ yields
(\ref{decompose}) as the marginal distribution of $J(n)$,
and simplifying yields (\ref{simpler 1}).
\end{proof}

\begin{lemma}\label{lemma j bound}
Let $H(n)$ have the harmonic distribution
(\ref{harmonic}) on 1 to $n$, so that for $i \leq n$,
$\p(H(n)=i) = 1/(i h_n)$.
For the $J(n)$ in Lemma \ref{density of J lemma}, based on a size
biased permutation of the natural infinite Poisson multiset of
prime powers,
\begin{equation}\label{dtv to H}
d_{TV}(J(n),H(n)\,) := \sum_{i \leq n} |\p(J(n)=i)-\p(H(n)=i)|
 = O \left( \frac{1}{\log n}\right).
\end{equation}
Furthermore, the relative error in approximating
the density of $J(n)$ by the harmonic density is $O(1/\log n)$,
uniformly:
\begin{equation}\label{point for j}
\max_{1 \leq i \leq n} \left| \frac{ \p(J(n)=i)}{\p(H(n)=i)} - 1 \right|
= O \left( \frac{1}{\log n}\right).
\end{equation}
\end{lemma}

\begin{proof}
Define $ d_n(i) := (i \log n) \p(J(n)=i)  $ so that  (\ref{simpler 1})
can be rewritten as
\begin{equation}\label{simpler 2}
d_n(i)  =  \  \sum_{q: \ \beta > 1-\alpha} \frac{1}{kq} \
  \int_{c>0} \beta  \ e^{-\beta c  } \  \
\frac{e^{-\alpha c} \log n }{\zeta(1+c/\log n)} \ dc.
\end{equation}
Since $\p(H(n)=i) = 1/(i h_n) = 1/(i \log n) \ (1+O(1/\log n))$,
showing
(\ref{point for j}) is equivalent to showing
$ \max_{1 \leq i \leq n} |d_n(i)-1|  =O(1/\log n)$.

In order to simplify (\ref{simpler 2}), we apply the
following
with $\delta = c/\log n$.
 From the well known $\zeta(1+\delta)=(1/\delta) + \gamma + O(\delta)$
as $\delta \rightarrow 0+$, we get $1/\zeta(1+\delta) =
\delta(1-\gamma \delta + O(\delta^2))= \delta - O(\delta^2)$
as $\delta \rightarrow 0+$.
It follows that $\exists C_1, |1/\zeta(1+\delta) - \delta|
\leq C_1 \delta^2$ for all $\delta > 0$.
This motivates us to consider a
first order approximation to the right side of
(\ref{simpler 2}), defined by
$$
e_n(i) :=   \sum_{q: \ \beta > 1-\alpha} \frac{1}{kq} \
  \int_{c>0} \beta  \ e^{-\beta c  } \  \
c \ e^{-\alpha c} \ dc.
$$
Our goal  is to show
that uniformly in $1 \leq i \leq n$, $e_n(i)=  1+ O(1/\log n)$; having
this, the
error estimate for $d_n(i)$ versus $e_n(i)$ will be  virtually the same
computation.

Note that since $\int_{c>0}c \, e^{-(\alpha+\beta)c} \ dc$
$=(\alpha+\beta)^{-2}$, and $ \beta = \log q/\log n$, we have
$$
e_n(i)= \frac{1}{\log n} \sum_{q>n^{1-\alpha}} \frac{\log q}{kq}
(\alpha+\beta)^{-2}.
$$
Instead of (\ref{1/q f}) we only need a crude bound, due to
Chebyshev, that
$$
R(x) := \sum_{q=p^k \leq x} \frac{\log q}{kq} \ -\log x =O(1).
$$
Fix $n$ and $i=n^\alpha$, $1 \leq i \leq n$, and define
$$
S_t := \sum_{n^{1-\alpha}<q \leq n^t} \frac{\log q}{kq} =
(t-(1-\alpha)) \log n - R(n/i) + R(n^t),
$$
so that by Abel summation
$$
e_n(t) = \frac{1}{\log n} \int_{t \in (1-\alpha,\infty)}  dS_t \ (\alpha+t)^{-2}
= \frac{1}{\log n}  \int_{(1-\alpha,\infty)}  S_t \  2
(\alpha+t)^{-3} \ dt.
$$
The contribution at infinity to the Abel summation is zero
since $S_t \sim t \log n$ as $t \rightarrow \infty$. From
$\sup_{x > 0} R(x) < \infty$, and
$\int_{(1-\alpha,\infty)} 2 (t+\alpha-1) $
$(\alpha+t)^{-3} \ dt
=1$ it follows that
$\max_{1\leq i\leq n} |e_n(i)-1| = O(1/\log n)$.

Finally, using $|1/\zeta(1+\delta) - \delta|
\leq C_1 \delta^2$, we have
$$
|d_n(i)-e_n(i)| \leq
\frac{C_1^2}{\log n}
  \sum_{q: \ \beta > 1-\alpha} \frac{1}{kq} \
  \int_{c>0} \beta  \ e^{-\beta c  } \  \
c^2  \ e^{-\alpha c} \ dc,
$$
$$
=\frac{C_1^2}{(\log n)^2} \int_{(1-\alpha,\infty)} \ dS_t \  2
(\alpha+t)^{-3} \ dt = O\left(\frac{1}{\log n} \right).
$$
\end{proof}

\subsection{Filling in the extra prime factor}\label{sect extra}

As in Lecture 1,
we take $P_0(n)$ to be one or prime (and not a prime power!),
such that $J(n) P_0(n) \leq n$, picking uniformly over the
$1+\pi(n/J(n))$ possibilities.

\ignore{For the sake of keeping track of the source of error, we also
consider
 $P_0^H(n)$ to be one or prime,
such that $H(n) P_0(n) \leq n$, picking uniformly over the
$1+\pi(n/H(n))$ possibilities.
WE WILL OMIT P_0^H, and use the density ratio bound for J versus H directly}

With the notation $f_n(i) := \p (J(n) P_0(n) = i)$, the total
variation distance in the next lemma is $d_{TV}( J \, P_0, \ N \ )=$
\begin{equation}\label{3 forms of dtv}
 \sum_{1 \leq i \leq n} \left(f_n(i)-\frac{1}{n}\right)^+
= \sum_{1 \leq i \leq n} \left(f_n(i)-\frac{1}{n}\right)^-
= \frac{1}{2} \sum_{1 \leq i \leq n} \left|f_n(i)-\frac{1}{n}\right|.
\end{equation}

\begin{lemma}\label{lemma jp bound}\label{lemma dtv 2}
The total variation
distance
between the distribution of $J(n) P_0(n)$ and the uniform
distribution satisfies
\begin{equation}\label{dist to N goal}
d_{TV} ( N(n),  \ J(n) \, P_0(n) \ ) =
O \left( \frac{\ll n}{\log n} \right).
\end{equation}
\end{lemma}
\begin{proof}

Since $P_0$ is one or prime, for $1 \leq m \leq n$,
\begin{equation}
\label{ph exact}
f_n(m):= \p(J(n) \, P_0(n)=m) = {\sum_{p|m}}^* \p\left(J(n)= \frac{m}{p}\right) \
\frac{1}{1+\pi(np/m)},
\end{equation}
where $\sum^*$ indicates that the index $p$  ranges over prime divisors of $m$
\emph{also allowing } $p=1$.  Using $1/(i \log n)$ as an approximation for
$\p(J(n)=i)$, we consider the simpler expression
$$
g_n(m):= \frac{1}{n} {\sum_{p|m}}^* \frac{1}{\log n} \, \frac{np}{m} \,
\frac{1}{1+\pi(np/m)}.
$$
We have
\begin{equation}
\label{ph simple}
\sup_{1\leq m \leq n} |1-f_n(m)/g_m(n)| = O(1/\log n)
\end{equation}
thanks to (\ref{point for j}),
and hence also
$\sum_{m \leq n} | f_n(m)-g_n(m)|$ $ = O(1/ \log n)$.
\ignore{ vestige of g^* versus g, from omitting the term  indexed by p=1
Omitting the contribution from $p=1$, we consider
\begin{equation}
\label{g}
g_n(m):= \frac{1}{n} \sum_{p|m} \frac{1}{\log n} \, \frac{np/m}{1+\pi(np/m)}.
\end{equation}
We have
$\sum_{m \leq n} | g_n^*(m)-g_n(m) |$
$= \frac{1}{n \log n} \sum_{m \leq n} \frac{n/m}{1+\pi(n/m)}$
$=O(1/\log n)$.
Hence
$\sum_{m \leq n} |f_n(m)- g_n(m)| = O(1/ \log n)$, and}
Thus
(\ref{dist to N goal}) is equivalent to
$\sum_{m \leq n} | g_n(m)-\frac{1}{n} |$ $ = O(\ll n/ \log n)$.

\def\li{{\rm li}}
  From the well known error bound for the prime number theorem,
$\pi(x)= \li(x)+O(xe^{-c\sqrt{\log x}}),$ together with the
approximation $\li(x) = (x/\log x)[1+1/\log x +O((\log x)^{-2})]$
we have
$$
r(x) := \left| \frac{x}{1+\pi(x)} - (\log x -1) \right| = O(1/\log x).
$$
Thus a good approximation to $g_n(m)$ will be given by
$$
 h_n(m) :=
 \frac{1}{n \, \log n}
{\sum_{p|m}}^* \left( \log \left( \frac{np}{m} \right) -1 \right).
$$
Writing $s(i)$ for the largest squarefree divisor of $i$,
and $\omega(i)$ for the number of distinct prime divisors of $i$,
we have
\begin{equation}
\label{def hfun}
h_n(m) =
 \frac{   \log(s(m))  +
 (1+\omega(m))  \left(  \log \left( n/m \right) -1 \right) }{n \, \log n}.
\end{equation}

With $N \equiv N(n)$ to represent the uniform distribution on
1 to $n$, the total variation distance in (\ref{dist to N goal}) is
approximately
\begin{equation}
\label{dist heuristic}
\frac{1}{2} \sum_{m \leq n} \left| h_n(m)-\frac{1}{n} \right|
= \frac{1}{2} \, \e \left|
\frac{\log(s(N))-\log n}{\log n} + \frac{ 1+\omega(N)}{\log n}
 \left(  \log \left( n/N \right) -1 \right) \right|.
\end{equation}

Before completing the proof of (\ref{dist to N goal}), we
outline the analysis to focus on the source of
the $\ll n$ factor in (\ref{dist to N goal}), and the reason that it cannot
be decreased.
The net contribution from the first term inside the expectation in
(\ref{dist heuristic}) is $O(1/\log n)$, and for
the second term, the two factors are approximately uncorrelated,
with $\e (1+\omega(N))/\log n \sim \ll n / \log n$ for the first factor.
The second
factor has $\e | \log \left( n/N \right) -1  | \rightarrow
\e  |S-1| = 2/e$, where $S$ has the standard exponential distribution,
with  $\p(S>t)=e^{-t}$.  Thus it \emph{should} be
possible to show that
\begin{equation}
\label{log log asymptotics}
d_{TV} ( N(n),  \ J(n) \, P_0(n) \ ) \sim \frac{1}{e}
 \left( \frac{\ll n}{\log n} \right).
\end{equation}

The estimates for the simpler task (\ref{dist to N goal})
in place of (\ref{log log asymptotics})
are as follows.
We need only $K := \sup_{x \geq 1} r(x) < \infty$ to conclude that
$|g_n(m) - h_m(n)|$ $\leq 1/(n \log n) \sum^*_{p|m} r(np/m)$
$\leq K (1+ \omega(m))/(n \log n)$, and hence
$\sum_{1\leq m \leq n} |g_n(m) - h_n (m)| $
$\leq K \e (1+\omega(N(n)))/\log n$
$=O(\ll n/ \log n)$.
In (\ref{dist heuristic}), for the first term
with $\e \log(s(N))-\log n$
we have $\e \log N(n) -\log n \rightarrow 1$
by Stirling's formula, and $\e \log (N(n)/s(N(n)))$
$\leq \sum_{p \leq n}$ $ \log p \, \e (C_p(n)-1)^+$
$\leq \sum_p \log p \sum_{k \geq 2} p^{-k} < \infty$.
For the second term, we apply Cauchy-Schwarz, with
$\e (1+\omega(N(n)))^2 \sim (\ll n)^2$ and
$\e (\log(n/N(n))^2 \rightarrow 1$.
These bounds combine to show that
$\sum_{m \leq n} | h_n(m)-\frac{1}{n}|$
$=O(\ll n / \log n)$.  Together with our previous
bounds comparing $f,g,$ and $h$, we have proved
(\ref{dist to N goal}).
\end{proof}

\ignore{It is  trivial that $d_{TV}(J(n) \, P_0(n), H(n) \, P_0^H (n)) \leq
d_{TV}(J(n),H(n))$. Combining (\ref{dtv to H}) with
(\ref{dist to N goal}) and using the triangle inequality
establishes
\begin{lemma}\label{lemma dtv 2}
\begin{equation}\label{dist to N goal}
d_{TV}( (J(n) \, P_0(n), \ N(n) \ )
=
O \left( \frac{\ll n}{\log n} \right).
\end{equation}
\end{lemma}
}

\subsection{Keeping score: \ 1 insertion and on average
$1+O((\ll n)^2/\log n)$ deletions
suffice for primes}

We need  to control the
expected number of deletions used to convert $M(n)$ to $N(n)$,
which correspond to the number of primes not exceeding $n$
but occurring in the size
biased permutation {\em after} the partial product $J(n)$.
A priori it seems reasonable to believe that one
would have to calculate something along the lines of
(\ref{feller bound}) for permutations,  crossed with
(\ref{decompose}) for primes.  Happily, the idea of
``matching intensity''
can be used to finesse the calculation.

 From (\ref{dist to N goal}) in Lemma \ref{lemma dtv 2}
it follows
(see section \ref{details} iff you want to know the details)
that on a single probability space we can
construct independent $Z_p$, together with $J(n)$ and
$P_0(n)$, and an exactly uniform $N(n)$,
so that always $J(n) | M(n)$, and the good event
\begin{equation}\label{good event N}
  E_n = \{ J(n) P_0(n) = N(n)\, \}
\end{equation}
has
\begin{equation}\label{dtv on E}
\p(E_n^c)=d_{TV}( (J(n) \, P_0(n), \ N(n) \ )
=
O \left( \frac{\ll n}{\log n} \right).
\end{equation}

On the uncoupled event, $E_n^c$, how many prime factors,
that might contribute to $d_W(n)$, can we
expect to see? Recall our
notation from section \ref{indel section}.  Lemma 6 in
\cite{primedw} states that for events $E$
of small, but not too small probability, the expected number of
prime factors is $O((\log \log n) \p(E))$.  The precise statement
is:
 uniformly in $\delta \in [0,1]$,
\begin{equation}\label{quote lemma 6}
  \sup_{E:\, \p(E) \leq \delta} \e(1(E)\ \Omega(N(n))\,)
 =  O \left( \max\left( \ \delta \ll n, {1 \over \log n} \right)
 \right).
\end{equation}
With $E = E_n^c$ and $\delta = \p(E_n^c)$, the combination of
(\ref{dtv on E}) with (\ref{quote lemma 6}) shows that
\begin{equation}\label{N on top1}
  \e \left(1(E_n^c) \Omega\left(  \frac{N}{(N,M)} \right)\right) \leq
 \e \left(1(E_n^c) \Omega\left(  N \right)\right) =
  O \left( \frac{(\ll n)^2}{\log n} \right).
\end{equation}
Now $J(n) | M(n)$ always, and on $E_n$ we have $N(n)=J(n)P_0(n)$
so that $N/(N,M) \ | P_0$, so that
\begin{equation}\label{N on top2}
   \e \left(1(E_n) \Omega\left(  \frac{N}{(N,M)} \right)\right) \leq 1.
\end{equation}
Adding gives
\begin{equation}\label{N on top}
  \e \Omega\left( \frac{N(n)}{(N(n),M(n))} \right) \leq
 1 + O \left( \frac{(\ll n)^2}{\log n} \right).
\end{equation}

Now {\em any coupling} has
$$
  \e  \Omega\left( \frac{N(n)}{(N(n),M(n))} \right)
- \e \Omega\left( \frac{M(n)}{(N(n),M(n))} \right)
= O(1/\log n),
$$
because (see e.g. \cite{tenenbaum} p. 41)
\begin{equation}\label{intensity match}
\e \Omega(N(n)) - \e \Omega(M(n)) = O(1/\log n).
\end{equation}
Combining this with (\ref{N on top})
yields
\begin{equation}\label{M on top}
  \e \Omega\left( \frac{M(n)}{(N(n),M(n))} \right) \leq
   1 + O \left( \frac{(\ll n)^2}{\log n} \right).
\end{equation}
Adding (\ref{N on top}) and (\ref{M on top}) proves
\begin{theorem}
\label{dw upper theorem}
The coupling of section \ref{coupling 2}, based on a size biased
permutation of the natural Poisson multiset of prime powers, and
extended to include a uniform random integer $N(n)$,
has
\begin{equation}\label{dw upper claim}
  \e \sum_{p \leq n} |C_p(n)-Z_p| \ \leq 2 + O\left( \frac{(\log
\log n)^2}{\log n} \right),
\end{equation}
and hence $d_W(n) \leq  2 + O\left( (\log
\log n)^2/\log n \right) $.
\end{theorem}

A separate argument (see \cite{primedw}) shows that, for {\em any
coupling},  with probability approaching one,
at least one insertion is necessary to convert $M(n)$ to $N(n)$,
i.e.
$1=\lim \p(  \sum_{p \leq n} (C_p(n)-Z_p)^+  \ \geq 1$.  In particular,
$1 \leq \liminf \e  \sum_{p \leq n} (C_p(n)-Z_p)^+ $.  Using
(\ref{intensity match}),
the average number of deletions
is within $O(1/\log n)$ of the number of insertions.  Hence
$1 \leq \liminf \e  \sum_{p \leq n} (Z_p-C_p(n))^+ $, so that
$2 \leq \liminf  \e \sum_{p \leq n} |C_p(n) - Z_p|$.  This shows that
that $\lim d_W(n)=2$, and that our
coupling, from the point of view of  the insertion-deletion metric in
section \ref{indel section},
is asymptotically optimal.

\subsection{Extending the coupling
to $N(n)$, constructively}\label{details}
(The reader is invited to skip past this section, which
defends the claim at (\ref{good event N}).)
Our coupling as described so far is fairly natural and
explicit.  It starts with independent Poisson $A_q$ for
$q=p^k$.
These determine the prime powers $Q_i$ and the $Z_p$
such that
 $M(n) := \prod_{p\leq n} p ^{Z_p}$ $=\prod_{i: Q_i=p^k, p\leq n} Q_i$,
 where the $Z_p$ for primes $p$ are
independent geometric.
Use independent exponentially distributed $S_1,S_2,\ldots$
to give a size biased permutation of these prime powers $q$;
this determines $J(n)$, a divisor of $M(n)$.  A single uniformly
distributed random variable $U$, independent of
the $Q_i$ and $S_i$ can then be used to determine
$P_0$, via the recipe:  with $K(n) := 1+\pi(n/J(n))$,
let $P_0 = 1$ if $K(n) U \leq 1$, and let
\begin{equation}
\label{use U}
P_0 = p_i \mbox{ if } \ K(n)U \in (i,i+1],
\end{equation}
where $p_i$ denotes the $i^{th}$ smallest prime.
   Equation (\ref{dist to N goal}) gives an
upper bound on the total variation distance
between the distributions of $J(n) P_0(n)$ and of $N(n)$,
and (\ref{3 forms of dtv}) emphasizes that this is just
about the {\em distribution} of $J(n) P_0(n)$.

We have
constructed $J(n) P_0(n)$ so far, and we have  {\em not yet }
constructed  $N(n)$.
It is a standard
coupling argument that  there exist couplings in which
the good event
$E_n = \{ J(n)P_0(n) = N(n) \}$
has the maximum possible probability,
with
$\p(E_n^c) = d_{TV}( J P_0, \ N \ )$.
It is also true, but less obvious,
that there exist such couplings which
extend our already given construction of  $((Q_i,S_i)_{i \geq 1},P_0)$.
We note further that joint distribution of $((Q_i,S_i)_{i \geq 1},P_0,N)$
for such an extension is {\em not uniquely determined}.
The next paragraph offers a constructive choice of
joint distribution.
\ignore{, together with some deterministic machinery
that can be used later when the Poisson process with intensity
$e^{-wy}\ dw \ dy$ is used as the starting point for a
more elaborate coupling}

We present a recipe for constructing $N(n)$ as
a function of the random variables used above, together with some
auxiliary randomization.
  Take two additional uniform random variables $U_{1},U_{2}$
with $U,U_{1},U_{2},Q_1,Q_2,\ldots,S_1,S_2,\ldots$
independent.  We define a deterministic
function
$r_n(u,u_{1},u_{2},q_1,q_2,\ldots,s_1,s_2,\ldots)$
such that for all $\omega \in \Omega$,
$N(n) := r_n(U,U_{1},U_{2},Q_1,Q_2,\ldots,S_1,S_2,\ldots)$
and $\p(N(n) \neq J(n)P_0(n)) = d_{TV}(J P_0,N)$.
The recipe $r_n$ is determined by two requirements.  First,
let $b_n(i) :=  \min(f_n(i),\frac{1}{n})/f_n(i)$, and let
$E_n$ be the event
\begin{equation}\label{def En}
E_n := \{ U_{1} \leq  b_n(J(n)P_0(n)) \}.
\end{equation}
Note that $\p (E_n) = \sum_{1\leq i \leq n} f_n(i) b_n(i) $
$=\sum \min (f_n(i),\frac{1}{n}) = 1 - d_{TV}( J P_0, \ N \ )$.
On the event  $E_n$, we {\em define} $N(n)$ by
$N(n)=J(n)P_0(n)$.
Second, let $G_n(j) := \sum_{i \leq j} (f_n(i)-\frac{1}{n})^- \
 / d_{TV}(J P_0,N)$,
so that by  (\ref{3 forms of dtv}),  $G_n(n)=1$.
On the event $E_n^c$, we {\em define} $N(n)$, to have one
of the values $i$ for which $f_n(i)< 1/n$, by setting:
\begin{equation}\label{construct N}
\mbox{ on } E_n^c, \ \ \ N(n)=j \ \mbox{ if and only if }
G_n(j-1)<U_{2} \leq G_n(j).
\end{equation}
It follows that for $i=1$ to $n$, $\p(N(n)=i) =1/n$, and that $E_n$
satisfies (\ref{good event N}) and \eqref{dtv on E}.

The above construction yields $N(n),$ uniformly distributed
from 1 to $n$, together with random integers $J(n)$ and $P_0(n)$
that evolve smoothly with $n$ growing, such that the event
$E_n = \{N(n)=J(n)P_0(n)\}$ has the maximal possible probability,
namely $1-d_{TV}( (J(n) \, P_0(n), \ N(n) \ )$.  Is it the case
that with probability one, for all sufficiently large $n$, we have
$N(n)=J(n) P_0(n)$?  This is not a trivial question, as
the events $E_n$ are not nested, and the sum of their probabilities
is infinite.

\begin{theorem}
\label{almost sure}
For the above construction,
$$
1=\p(E_n \mbox{ eventually }).
$$
\end{theorem}

\begin{proof}
First note that as events, $\{ J(n) \rightarrow \infty \} = \{ \sum Z_p =
\infty\}$,
and hence $1=\p(J(n) \rightarrow \infty)$.
Recall our use of two fixed uniform random variables,
$U$ in \eqref{use U} and $U_1$  in \eqref{def En}.
For $\delta >0$ we will show that
\begin{equation}
\label{eventual task}
 \{U_1 < 1-\delta,\  U  > \delta , \mbox{ and } J(n) \rightarrow \infty\}
\subset  \{ E_n \mbox{ eventually } \}
\end{equation}
and hence $\p(E_n$ eventually)
$ \geq (1-\delta)^2$
.

Assume we are given an outcome in the event on the left side of
\eqref{eventual task}.  Since the $J(n)$ are partial products,
 $J(n) \rightarrow \infty$
ensures that $\omega(J(n)) \rightarrow \infty$ as $n \rightarrow
\infty$.
To have $E_n$ fail, we must have
$ b_n(J(n)P_0(n)) < 1-\delta$, and hence $n
f_n(J(n)P_0(n))>1/(1-\delta)>1+\delta$.  For $n$ and $\omega(m)$ both large,
arguing as in
the proof of Theorem \ref{lemma jp bound},
$f_n(m)/h_n(m) \rightarrow 1$ so that $f_n(m)>1+\delta$ implies
that $n h_n(m) > 1+\delta/2$. [The hypothesis that $\omega(m)$ is large
is needed to ensure that terms of $g_n(m)$ having $x=np/m$ small,
where we cannot guarantee that
$x/(1+\pi(x))$  is close to
$ (\log x -1)$, make a relatively negligible contribution.]
  Using only the bound
$\log s(m) \leq \log n$ in \eqref{def hfun}, this implies
that for sufficiently large $n$ and $\omega(m)$,
$(1+\omega(m)) \log (n/m) > (\delta/2) \log n$.
Pick $x_0 > 1/(2\delta)$ and large enough that
$x>x_0$ implies $\pi(2 \delta x)
>(\delta) (1+\pi(x))$.
Since $\sup_{1\leq i \leq n} \omega(i) = o(\log n)$,
$(1+\omega(m)) \log (n/m) > (\delta/2) \log n$ implies
that for sufficiently large $n$, $n/m >x_0$.
Thus for sufficiently large $n$, if $E_n$ fails then
$x = n/J(n)>x_0$,
But $U > \delta$ now implies $P_0(n) \geq 2 \delta n/J(n)$, which
would contradict $n/m > x_0$ with $m=J(n)P_0(n)$.
This shows that for the given
outcome, there is an $N_0$, (depending on the outcome
through the values of $J(1),J(2),\ldots$ and $U,U_1$,)
such that for all $n>N_0$, $E_n$ occurs.
\end{proof}

\section{Lecture 4: The distance to the  Poisson-Dirichlet}
\label{distance}

For an integer $N(n)$ distributed uniformly from 1 to $n$,
write
$$
N(n) = P_1(n) P_2(n) \cdots P_{K_n}(n) \ = P_1(n) P_2(n)
\cdots, \ \ P_1 \geq P_2 \geq \cdots,
$$
where $K_n = \Omega(N(n))$ is the number of prime factors of $N$,
and every $P_i(n)$ is either one or prime.
Billingsley  \cite{billing large} in 1972 proved that the
Poisson-Dirichlet process gives the limit in distribution for
the sizes of the large prime factors,
\begin{equation}
\label{Billingsley pd}
\left({\log P_1(n) \over \log n},{ \log P_2(n) \over \log n} ,\ldots \right)
\Rightarrow (V_1,V_2,\ldots),
\end{equation}
where $(V_1,V_2,\ldots)$
has the Poisson-Dirichlet distribution with parameter 1.
The marginal distribution of the largest component is given by
Dickman's \cite{dickman} function $\rho$, in the form
$\p(V_1 \leq 1/u) = \rho(u)$; see \cite{tenenbaum} Chapter III.5.
The characterization of
the limit which is useful for us is
\begin{equation}\label{PD as spacings}
(V_1,V_2,\ldots) \indist \mbox{RANK}(1-X_1,X_1-X_2,X_2-X_3,\ldots)
\end{equation}
where RANK is the function which sorts the coordinates in
nonincreasing order, and $X_1,X_2,\ldots$ are those points of the
scale invariant Poisson process which fall in $(0,1)$, indexed
with $1>X_1>X_2>\cdots >0$.  For other characterizations of
the Poisson-Dirichlet, see for example \cite{T1,pityor97}.
Donnelly and Grimmett \cite{DG}
gave a very nice proof of (\ref{Billingsley pd}) by showing that
a size biased permutation of the left side of (\ref{Billingsley pd})
converges in distribution to $(1-X_1,X_1-X_2,X_2-X_3,\ldots)$, and
then using the continuity of RANK on the simplex.

We asked:
how close are the right and left sides of  (\ref{Billingsley pd})?
One notion of approximation, from Knuth and Trabb Pardo
\cite{knuth}, is that for fixed $i$ and $t \in (0,1)$,
as $n \rightarrow \infty$,
$$
\p\left(\frac{\log P_i(n)}{\log n} \leq t\right) = \p(V_i \leq t) +
O\left( \frac{1}{\log n} \right).
$$
They also give a version with a $O(1/\log n) $ correction term, of the
form:  for fixed $i \geq i$, and fixed $t \in (0,1)$,
$\p(\log P_i(n)/(\log n)  \leq t)$ $ = \p(V_i \leq t)$ $ + r_i(t)/\log n +
O(1/(\log n)^2)$.
That a similar result holds for the \emph{joint} finite
dimensional distributions, together with an expansion in
negative powers of the logarithm, has recently been shown by
Tenenbaum \cite{rate bill}. To state this, for $k \geq 1$
write $F_n(\alpha_1,\ldots,\alpha_k) = \p( \log P_i(n)/(\log n) \geq \alpha_i$
for $i=1$ to $k$), and
$\phi_0(\alpha_1,\ldots,\alpha_k) =  \p(V_1\geq \alpha_1,\ldots,V_k \geq
\alpha_k)$,
 so that Billingsley's result (\ref{Billingsley pd}) is equivalent to:
for all $k \geq 1$ and $\alpha_1,\ldots,\alpha_k \in (0,1)$,
 $F_n(\alpha_1,\ldots,\alpha_k)=$
$\phi_0(\alpha_1,\ldots,\alpha_k) + o(1)$.  Tenenbaum's result is the following:
for every $k \geq 1$  there exist functions
$\phi_1,\phi_2,\ldots$,
continuous except  on finitely many hyperplanes, such that,
\begin{equation}
\label{joint tail}
F_n(\alpha_1,\ldots,\alpha_k)= \sum_{0 \leq h \leq H}
\frac{\phi_h(\alpha_1,\ldots,\alpha_k)}{(\log n)^h}
+ O\left( \frac{1}{(\alpha_k\log n)^{H+1}} \right)
\end{equation}
holds for all fixed $(\alpha_1,\ldots,\alpha_k)$, and
uniformly outside an exceptional set of  $(\alpha_1,\ldots,\alpha_k)$
with measure $O(\log \log n/ \log n)$,
specified in \cite{rate bill}

A very natural and useful choice of metric is the $l_1$ distance,
since this controls the approximation of the set of logarithms
of all divisors, see \cite{divisors,fractale},
by its limit, see \cite{dimacs}, section 22.
Approximations in this metric are not comparable to results such
as \eqref{joint tail}; an analogous situation is that, from
knowledge that the difference between the two sides of
\eqref{eq1.2} is at most $1/n$, Kubilius' fundamental lemma
\emph{does not} follow as a consequence.

A very natural and useful choice of metric is the $l_1$ distance,
since this controls the approximation of the set of logarithms
of all divisors, see \cite{divisors,fractale},
by its limit, see \cite{dimacs}, section 22.
For the metrized approximation question,
it is not necessary to divide by $\log n$.
For a proof or disproof of the following conjecture
about the $l_1$ distance,
 from \cite{AMS1}, I now offer a one hundred dollar prize.
\begin{conjecture} {\bf (\$100 prize offered) }
\label{100 conjecture}
For all $n \geq 1$, it is possible to construct
$N(n)$ uniformly distributed from 1 to $n$,
and the Poisson-Dirichlet process $(V_1,V_2,\ldots)$,
on one probability space, so that
\begin{equation}
\label{pd distance 1}
\e \sum |\log P_i(n) - (\log n) V_i | = O(1).
\end{equation}
\end{conjecture}

Note that the liminf of the left side of (\ref{pd distance 1}) is
at least 1, because $\e \sum \log P_i(n)$ $=\e \log N(n)$
$= \log n -1 + o(1)$, from Stirling's
$ n! \sim (n/e)^n \sqrt{2 \pi n}$,
while $\e \sum (\log n) V_i =\log n$.
We finished our workshop lectures by describing a coupling that
achieves $O(\ll n)$ in place of $O(1)$ in (\ref{pd distance 1});
here in the writeup we also present the proof that this coupling
works as claimed.

\begin{theorem}
\label{thm for PD}
The coupling of the Poisson-Dirichlet with a random integer
$N(n)$ uniformly distributed from 1 to $n$, described in section
\ref{sect PD primes}, achieves
\begin{equation}
\label{pd distance 2}
\e \sum |\log P_i(n) - (\log n) V_i | = O(\log \log n).
\end{equation}
\end{theorem}

Historical notes: \emph{all} logarithmic combinatorial structures
have a Poisson-Dirichlet limit for the fractions of system
size in the largest, second largest, third largest, $\ldots,$
components.
This was shown in 1977 for permutations,  by Kingman \cite{kingman77},
and independently by
Vershik and Schmidt \cite{vershik1,vershik2}.  It was shown for random
mappings --- where the Poisson-Dirichlet limit has parameter $1/2$
--- by
Aldous \cite{aldous mapping} in 1983.  It was shown
for a wide class of
combinatorial structures
by Hansen  \cite{hansen} in 1994,
and with local limit bounds, for a very
general scheme, in  \cite{abtlocal}.

The analog of Conjecture \ref{500 conjecture} is {\em false} for
permutations, for the simple reason that $n V_i$ has a distribution with
a continuous density, while the size $L_i(n)$ of
the $i^{th}$ largest cycle has integer support, so that under {\em
any conceivable} coupling, for every $i$,
$\liminf_n \e |L_i(n) - n V_i| \geq 1/4$.
For $i >(1+\epsilon) \log n$, one can match
$n V_i$ with $L_i(n)=0$, and the net result is that
$\liminf_n (1/\log n)
\e \sum_i |L_i(n) - n V_i| \geq 1/4$.  It is shown in \cite{four couplings}
that the coupling for
permutations which is analogous to our coupling in section
\ref{sect PD primes} achieves this lower bound, with
\begin{equation}
\label{perm PD distance}
\e \sum |\log L_i(n) - n V_i | \sim \frac{1}{4} \log n.
\end{equation}

\subsection{Growing a random integer from the Poisson-Dirichlet}
\label{sect PD primes}

This section gives a third coupling for growing a random integer
$J(n) P_0(n)$ which is close in distribution to
the uniform random integer $N(n)$.  Our first coupling, in Lecture 1,
determined $J(n)$ from a size biased permutation of the multiset with $Z_p$
spacings of size $\log p$, with $Z_p$ geometrically distributed.
Our second coupling, in Lecture 3, used instead a size biased
permutation of the Poisson multiset with $A_{p^k}$ spacings  of
size $\log p^k$, with $A_{p^k} \sim$ Poisson($1/(kp^k))$.
This Poisson multiset can be constructed by applying the
deterministic
function $h$ at (\ref{modified h})  to the points
of the scale invariant Poisson process, so the second coupling
may be viewed as constructing $J(n)$ from a deterministic function
of the scale invariant Poisson, together with the auxiliary
randomization of a size biased permutation.

For our third coupling, this Poisson multiset,
{\em and its permutation},
will be given by a deterministic function applied to
the
{\em spacings}
 of the Poisson-Dirichlet process.  These spacings are the points of
the scale invariant Poisson process, in order of a size biased
permutation.
The sizes  $\log p^k$ are slightly different
 from the sizes of the Poisson-Dirichlet spacings, so
the ordering of spacings in our third
coupling is a perturbation of that in our second coupling.

Start with the Poisson process on $(0,\infty)^2$, with
intensity $e^{-wy} \ dw \ dy$, from section
\ref{sect exp(-wy)}, with points $\{ (W_i,Y_i), i \in \BZ \}$,
but reverse the direction of the indexing, so that
$W_i < W_{i+1}$ for all $i \in \BZ$.  Defining $X_i :=
\sum_{j \geq i} Y_j$ gives the scale invariant Poisson
process $\X = \{ X_i: i \in \BZ\}$, indexed in decreasing order,
with $X_{i+1}<X_i$.  Now {\em shift} the indexing of
 $\{ (W_i,Y_i), i \in \BZ \}$, so that $X_1$
appears as the first point to the left of $\log n$.  To summarize,
we have the points of the scale invariant Poisson process $\X$,
indexed
so that
\begin{equation}\label{first index}
0< \cdots < X_3 < X_2 < X_1 < \log n \leq X_0 < X_{-1} < \cdots \ .
\end{equation}

The Poisson-Dirichlet from (\ref{PD as spacings}),
scaled up by $\log n$, is realized as
\begin{equation}\label{log n times PD}
  ((\log n)V_1, (\log n) V_2,\ldots) = \mbox{RANK}(\log n -
  X_1,X_1-X_2,\ldots).
\end{equation}

The points $Y_i = X_i-X_{i+1}$
for $i \in \BZ$ are the points of the scale invariant Poisson
process, and using $h$ from (\ref{modified h}), we
construct
random $Q_i$ by
\begin{equation}\label{Q from X}
 \ \ \ \ \log Q_i = h(Y_i), \ \ \ \mbox{with } Y_i=X_i-X_{i+1}, \ \ i \in \BZ.
\end{equation}

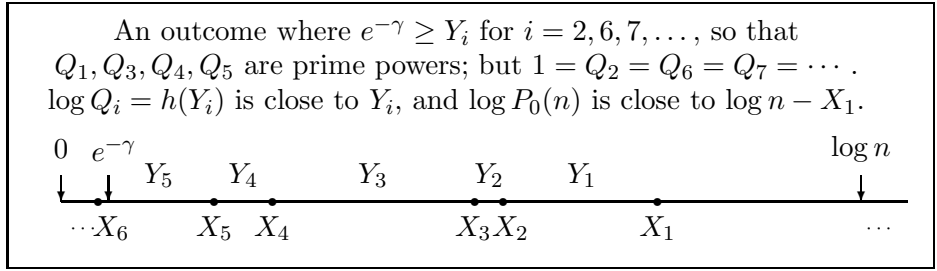
\begin{figure}[h]
\begin{picture}(400,100)(-20,0)
\thinlines \put(-20,0){\framebox(350,100)}
\put(-5,60){\shortstack{An outcome where $e^{-\gamma} \geq Y_i$
for $i=2,6,7,\ldots,$ so that \\ $Q_1,Q_3,Q_4,Q_5$ are prime
powers; but $1=Q_2=Q_6=Q_7=\cdots .$
 \\ $\log Q_i=h(Y_i)$ is close to $Y_i$, and $\log P_0(n)$ is
close to $\log n - X_1$.}}
\thicklines
\put(0,25){\line(1,0){320}}
\thinlines
\put(0,35){\vector(0,-1){10}}
\put(18,35){\vector(0,-1){10}}
\put(14,25){\circle*{3}}
\put(58,25){\circle*{3}}
\put(80,25){\circle*{3}}
\put(156.48,25){\circle*{3}}
\put(167.27,25){\circle*{3}}
\put(225.65,25){\circle*{3}}
\put(302.6,35){\vector(0,-1){10}}
\put(9,15){\makebox(0,0){\tiny$\cdots$}}
\put(310,15){\makebox(0,0){\tiny$\cdots$}}
\put(0,45){\makebox(0,0){0}}
\put(20,45){\makebox(0,0){$e^{-\gamma}$}}
\put(19,15){\makebox(0,0){$X_6$}}
\put(58,15){\makebox(0,0){$X_5$}}
\put(80,15){\makebox(0,0){$X_4$}}
\put(156,15){\makebox(0,0){$X_3$}}
\put(170,15){\makebox(0,0){$X_2$}}
\put(225.65,15){\makebox(0,0){$X_1$}}

\put(37,35){\makebox(0,0){$Y_5$}}
\put(69,35){\makebox(0,0){$Y_4$}}
\put(118,35){\makebox(0,0){$Y_3$}}
\put(162,35){\makebox(0,0){$Y_2$}}
\put(197,35){\makebox(0,0){$Y_1$}}
\put(302.6,45){\makebox(0,0){$\log n$}}
\end{picture}
\end{figure}

[The following information is motivational, and not part of
our proof.
Like the $Q_1,Q_2,\ldots$ in section \ref{coupling 2}, for every $q=p^k$
we have that $A_q := \sum_{i \in \BZ} 1(Q_i=k)$ is
Poisson($1/(kq)$),
with the $A_q$ mutually independent, but there are
several differences in the indexing scheme:  here the
$Q_i$ are indexed from right to left, there may be
instances of $Q_i=1$  in between occurrence of proper
prime powers, and most significantly, the sizes
used for the size biased permutation are the $X_i-X_{i+1}$ and not
the $\log Q_i$.  We will need to show that since
$h$ is close to the identity
function, the second and third couplings are close, in that
with high probability, they produce the same $J(n)$.]

Define $J^*(n)$ by
\begin{equation}\label{J 3}
J^*(n) = \prod_{i \geq 1} Q_i.
\end{equation}
It is conceivable that $J^*(n)>n$ if $X_1$ is close enough to $\log n$
and $h$ gets applied in places with $h(y)>y$; in this case we will
prescribe $P_0^*(n)=1$.   When $J^*(n)\leq n$, take $P_0^*(n)$
to be one or prime,
such that $J^*(n) P_0^*(n) \leq n$, picking uniformly over the
$1+\pi(n/J^*(n))$ possibilities.

We have two tasks.  The first, carried out in
Lemma \ref{task 1}, is to show
that the prime factors $P_1^*(n),P_2^*(n),\ldots$
of this random
integer $J^*(n) P_0^*(n)$,
listed in nonincreasing order, give a vector of logarithms
close to the Poisson-Dirichlet.
The second task,
carried out in Lemma \ref{task 2}, is to show,
like Lemma \ref{lemma dtv 2}, that the random integer
we have constructed is close to uniform.  From
these two lemma, Theorem \ref{thm for PD} follows easily.
We will state the lemmas, then give the proof of Theorem
\ref{thm for PD} to finish this section.  The next section
provides
the proofs of the two lemmas.

\begin{lemma}
\label{task 1}
\begin{equation}\label{dist to PD 1}
\e \sum_{i \geq 1} |\log P_i^*(n) - (\log n) V_i | = O(1).
\end{equation}
\end{lemma}

\begin{lemma}
\label{task 2}
\begin{equation}\label{dist to N 3}
d_{TV}( (J^*(n) \, P_0^*(n), \ N(n) \ )
=
O \left( \frac{\ll n}{\log n} \right).
\end{equation}
\end{lemma}

\noindent {\bf Proof of Theorem \ref{thm for PD}}
As in section \ref{details}, the coupling of $J^*(n)P_0(n)$ with
the Poisson-Dirichlet process can be extended to include $N(n)$,
distributed uniformly on 1 to $n,$
in such a way that
$N(n)=J^*(n)P_0^*(n)$ except on a
``bad'' event of probability equal to the total variation distance
$d_{TV}( (J^*(n) \, P_0^*(n), \ N(n) \ )$.
On the bad event we have no control over the $l_1$
distance, apart from the trivial bound: it is at most $ 2 \log n$.  Multiplying
by an upper bound on the probability of the bad event, from
Lemma \ref{task 2},
gives us the main contribution to the error, of size
$O(\log \log n)$.  On the complement of the bad event, the
contribution is $O(1)$ by Lemma \ref{task 1}.  Adding these errors
gives (\ref{pd distance 2}). \qed

\subsection{Proofs of the lemmas for Theorem \ref{thm for PD}}
~\\~\\
\noindent {\bf Proof of Lemma \ref{task 1}}
There are three contributions to the $l_1$ distance:
first $Y_i$ versus $\log Q_i$,
second a  contribution from $Q_i$
which are prime powers but not prime,
and third $\log P_0^*(n)$ versus $(\log n - X_1)$.
Observe that for any vectors ${\bf x}=(x_1,x_2,\ldots)$
and ${\bf y}=(y_1,y_2,\ldots)$ in $[0,\infty)^\BN \cap l_1$,
the $l_1$ distance is not increased if the coordinates
of both vectors are sorted:  $|| $RANK(${\bf x}) -
$RANK(${\bf y})||_1
\leq || {\bf x} - {\bf y}||_1$.

Consider the random variable
\begin{equation}\label{def D}
D =  \sum_{i \in \BZ} |h(Y_i) - Y_i|.
\end{equation}
It has
\begin{equation}\label{E D}
  \e D
=\e  \sum_{i \in \BZ} |h(Y_i) - Y_i|
 = b_0 < \infty,
\end{equation}
using (\ref{area bound q}) and the scale invariant spacing lemma.

The first contribution to our expected $l_1$ distance is handled by:  for all $n$,
\begin{equation}\label{h is easy}
  \e \sum_{i \geq 1} |\log Q_i - Y_i|
=\e  \sum_{i \geq 1} |h(Y_i) - Y_i|
\leq \e D = b_0 < \infty.
\end{equation}

To handle the second contribution,
we ``split up'' any prime
powers $p^k$ with $k>1$ which may occur among the $Q_i$, defining
$Q_1^*,Q_2^*,\ldots$ to be one or  prime,
so that $J(n)=\prod_{i \geq 1} Q_i$ $=\prod_{i \geq 1} Q_i^*$.
To do this, start with  $Q_j^*=1$ whenever $Q_j=1$; some of these will be
changed. Always take $Q_i^*= p$
when $Q_i=p^k$.  For any $Q_i=p^k$ with $k>1$, take
$k-1$ indices $j$ for which $Q_j^*=1$ and change these to
$Q_j^*=p$.  With ${\bf x} = (\log Q_1,\log Q_2,\ldots)$
and ${\bf y}= (\log Q_1^*,\log Q_2^*,\ldots)$ we have
$ || {\bf x} - {\bf y}||$
$= \sum_{p,k,i} 2(k-1)(\log p)1(Q_i=p^k)$
$\leq \sum_{q=p^k,k>1}2 (\log q) A_q$.
Note that $\e  \sum_{q=p^k,k>1} (\log q) A_q$
$= \sum_{q=p^k,k>1} (\log q) /(kq) < \infty$.

To handle the third contribution:
using $c = (\log n - \sum h(Y_i))^+$, our recipe for
$P_0^*(n)$ is to choose uniformly over the $1+\pi(e^c)$ numbers
which are one, or a prime at most $e^c$.  Writing $\p_c$ and
$\e_c$ for such
a choice of $P_0$, we have $\sup_{c \geq 0} \e_c(c-\log P_0) < \infty$,
as a simple consequence of
the prime number theorem, that
$\pi(e^x) \sim e^x/x$ as $x \rightarrow \infty$.
Using $X_1=\sum_{i \geq 1} Y_i$, we have
$|(\log n - X_1) - (\log n - \sum_{i \geq 1} h(Y_i)|$
$= | \sum_{i \geq 1}(Y_i~-~h(Y_i) |$,
with expectation bounded by $b_0$, using (\ref{E D}).
Combining yields $\sup_n \e | \log P_0^*(n)- (\log n -X_1)| <
\infty$.

Combining these three contributions, and using the $l_1$
contraction property of the  function RANK,
proves (\ref{dist to PD 1}).
\hfill \qed

\noindent {\bf Proof of Lemma \ref{task 2}}

In contrast to (\ref{first index}) and (\ref{Q from X}), we
now re-index the $\{ Y_i: i \in \BZ\}$
so that
$$
\cdots < Y_{-2} < Y_{-1}<Y_0 \leq e^{-\gamma} <Y_1<Y_2<\cdots \ .
$$
The choice of location of $e^{-\gamma}$
has the effect that $ Q_1,Q_2,\ldots $ are
the primes and prime powers, while $1=Q_i$ for $i \leq 0$.
The indexing of the $Y_i$ in the order of their own values has the
effect that, with
$$
S_i := W_i/Y_i,
$$
the $S_i$ for $i \in \BZ$
are, conditional on the values of
$Y_i, i \in \BZ$,  mutually independent, standard exponentials ---
and this would not
have been true under the previous indexing, where $W_i < W_{i+1}$.

We construct the second coupling, of section \ref{coupling 2},
 from this multiset $\{ Q_1,Q_2,\ldots \}$.
For the size biased
permutation of the $\log Q_i$, we use the exponentially
distributed labels
\begin{equation}\label{s order}
\hat{W_i} := S_i/\log Q_i = W_i \ \frac{Y_i}{h(Y_i)},
\end{equation}
so that $J(n)$ is the largest partial product of
the $Q_i$ not exceeding $n$, with the $Q_i$ taken in order of
decreasing labels  $\hat{W_i}$.
In contrast,  $J^*(n)$ is  a partial product of $Q_1,Q_2,\ldots$
taken in order of decreasing $W_i$.
Because $h$ is close to the identity, eventually the permutation
induced by the $W_i$ agrees with the permutation induced by the
$\hat{W_i}$.  We will show that
\begin{equation}\label{2 vs 3}
  \p ( J^*(n) \neq J(n)) \ = O\left( \frac{\log\log n}{\log n} \right),
\end{equation}
and thus $d_{TV}(J^*(n)P_0^*(n), J(n)P_0(n)) = O(\log \log n /\log n)$.
Combined with Lemma \ref{lemma dtv 2}, this yields
 $d_{TV}(J^*(n)P_0^*(n), N(n)) = O(\log \log /\log n)$.
As a remark, we believe that the quantity in (\ref{2 vs 3}) is
actually $O(1/\log n)$, but since this improved bound would not
improve the overall result, we settle for the looser bound.

First, we show that the effect in (\ref{J 3}) of stopping
at the largest $X_i < \log n$, rather than stopping with the
largest partial product not exceeding $n$, is negligible.
Consider the ``good'' event
\begin{equation}\label{good event}
  G = \{  D \geq \log \log n \mbox{ or }
\X \cap (\log n - \log \log n,
\log n + \log \log n) \neq \emptyset \
   \},
\end{equation}
with $D$ given by (\ref{def D}).
We will show that
\begin{equation}\label{upper for good}
  \p (G^c) = O\left( \frac{\log \log n}{\log n} \right),
\end{equation}

To show this, first observe that $\p (\X \cap (\log n - \log \log n,
\log n + \log \log n) \neq \emptyset)$ $=O(\log \log n / \log n)$,
since the intensity of $\X$ is $1/x \ dx$.
Second, observe that
$\p(D \geq \log \log n)$ $=O(1 / \log n)$,
which follows from showing $\e e^{\beta D} < \infty$
with some $\beta >1$.  In fact, $\e e^{\beta D} < \infty$
for all $\beta$; with $g$ as defined following
(\ref{1/q f}), and using (\ref{exponential map}),
 we have $D=\sum_{i \in \BZ} |g(L_i)-\exp( L_i)|$
so that
$$
\e e^{\beta D} = \exp\left(
\int_{-\infty}^\infty (e^{\beta (g(l) -e^l)}-1) dl
 \right).
$$
The contribution to the integral from the neighborhood of
$-\infty$ is finite, using $g(l)=0$ there, and
contribution to the integral from the neighborhood of
$\infty$ is finite, using $g(l)-e^l=O(e^l \exp(-c e^{l/2}))$
as $l \rightarrow \infty$, which follows from (\ref{1/q f}).

Next, we consider a bad event on which the two permutations,
one induced by the sizes $W_i$, and the other induced by the sizes
$\hat{W_i}:= W_i \ {Y_i}/{h(Y_i)}$ as in (\ref{s order}), might disagree
in a nontrivial way,  i.e. giving the opposite ordering to
$Q_i,Q_j>1$,
out beyond the place where the partial sum of the $Y_i$ exceeds
$(\log n)/2$.
Let
\begin{equation}\label{bad}
  B = \{
\exists i\neq j, \ Y_i,Y_j>e^{-\gamma}, W_i<W_j,  \hat{W_i} \geq
 \hat{W_j}, \ \sum Y_k 1(W_k \geq W_i) > (\log n)/2
  \}.
\end{equation}
For $n$ so large that $\log \log n < (\log n)/2$, we have
$$
\{ J^*(n) \neq J(n) \}\subset G^c \cup B,
$$
so that it only remains to show $\p(B) =O(\log \log n / \log n)$.

Let
$$
T_w = \sum_{i \in \BZ} Y_i 1(W_i > w).
$$
The distribution of $T_w$ is exponential, with $\e T_w=1/w$
--- see for example \cite{kingman}, under the ``Moran process'', or
\cite{primedw}, where this is used as an ingredient in the proof of the
scale invariant spacing lemma.
We say that $((w,y),(w',y'))$ is a ``potential witness''
to the bad event $B$ if
$y,y' > e^{-\gamma}$,$w'<w$,
the Poisson process
$\{ (W_i,Y_i) \}$
has points at $(w,y)$ and $(w',y')$, and no points
$(W_k,Y_k)$ with $w'<W_k<w$, and $T_w + y + y' > (\log n)/2$,
and
\begin{equation}\label{witness}
\log(\frac{w}{w'}) \leq |\log(y/h(y))| + |\log(y'/h(y')|.
\end{equation}
Let $N_B$ denote the number of potential witnesses, and observe
that $B \subset \{ N_B > 0\}$.  Thus, it only remains to calculate
that $\e N_B = O(\log \log n/\log n)$; and in fact we will show
that it is $O(1/\log n)$.

Now $\e N_B$ is merely a four-fold integral, so the reader is
invited to take our claim at face value; but for those
declining our invitation, here are the details.
\ignore{
First note
$$
\e N_B =  \int_{-\infty <w'<w<\infty} \int_{y,y'>e^{-\gamma}}
e^{-wy} \ dw \ dy \ e^{-w'y'} \ dw' \ dy'
\log(w'/w) \
\p(y+y'+T_w > \log n/2) \
1(\log(\frac{w}{w'}) \leq |\log(y/h(y))| + |\log(y'/h(y')|).
$$
a little too long for one line!
}
Conditional
on having points $(W_i,Y_i)$ and $(W_j,Y_j)$ with $W_i=w,W_j=w'$,
the joint distribution of $Y_i,Y_j,T_w$ is that of three
independent exponentials, with means $1/w,1/w'$, and $1/w$
respectively.  Let $c_0 := \exp(2\sup \{ |\log(y/h(y))|\! \!:
 y >e^{-\gamma}\ \} )$,
so that for a potential witness, $w/w' \leq c_0$, and $Y_j$ lies below
an exponential with mean $c_0/w$.
This gives us, for the conditional probability, that
 $\p_{w,w'} (Y_i+Y_j+T_w > \log n / 2)$
$\leq 3 \p((c_0/w)S_1 > \log n / 6)$
$=3 \exp(-w \log n/(6c_0))$.

We need some monotonicity for the next simplification.
 from (\ref{1/q f}) we have that
$$
h(x) = x + O(x e^{-c\sqrt{x}})  \ \ \ \mbox{ as } x \rightarrow
\infty,
$$
so that for some constants $c_1,c_2>0$, for all $y > e^{-\gamma}$,
$|\log(h(y)/y)| < c_1  e^{-c_2 \sqrt{y}}$.  Thus we can relax
the notion of ``potential witness,'' replacing (\ref{witness})
with the condition
\begin{equation}\label{witness'}
\log(\frac{w}{w'}) \leq  r(y) + r(y'), \ \ \mbox{ where } r(y) =
 c_1  e^{-c_2 \sqrt{y}} .
\end{equation}
Write $N_R$ for the number of potential witnesses in this relaxed
sense,
so that $N_B \leq N_R$.
Now the indicator of the inequality (\ref{witness'})
is a {\em decreasing} function of $(y,y')$, while
the indicator $1(y+y'+t>\log n/2)$ is an
{\em increasing} function, so that we have negative
correlations (with respect to $Y_i,Y_j,$ and $T_w$, which are
conditionally independent given $w,w'$):
$$
\p_{w,w'} \left(Y_i+Y_j+T_w > \frac{\log n}{ 2}, \log(\frac{w}{w'}) \leq
r(Y_i)+r(Y_j)\right) \hfill
$$
$$\hfill \leq
\p_{w,w'} \left(Y_i+Y_j+T_w > \frac{\log n}{ 2}\right)
\p_{w,w'} \left( \log(\frac{w}{w'}) \leq
r(Y_i)+r(Y_j)\right).
$$
We use the monotonicity of $r(\cdot)$ again,  to justify
$$
\p_{w,w'} \left( \log(\frac{w}{w'}) \leq r(Y_i)+r(Y_j)\right)
\leq \p \left( \log(\frac{w}{w'}) \leq
r(\frac{S_i}{w})+r(\frac{S_j}{w})\right),
$$
where $S_i,S_j$ represent independent standard exponentials.

Recalling that $\{ W_k: k \in \BZ \}$ form a copy of the scale
invariant Poisson process $\X$, and $\p(\X \cap (w',w)=\emptyset) =
w'/w$,
$$
 \e N_R
\leq \int \! \! \int_{w'<w} \frac{dw'}{w'}\ \frac{dw}{w}
\ \frac{w'}{w}   \ 3 \exp(-w \log n/(6c_0)) \
\p \left( \log(\frac{w}{w'}) \leq
r(\frac{S_i}{w})+r(\frac{S_j}{w})\right).
$$
Recall further that for two consecutive points $W_j<W_i$
of the scale invariant Poisson process, conditional on
$W_i=w$ the distribution of $W_j$
is that of $Uw$, where $U$ uniformly distributed in $(0,1)$.
Thus the right hand side above is equal to
$$
\int_{w>0}
\frac{dw}{w} \ 3 \exp(-w \log n/(6c_0)) \
\p \left( \log(\frac{1}{U}) \leq
r(\frac{S_i}{w})+r(\frac{S_j}{w})\right).
$$
Since $-\log U$ is exponentially distributed, with density bounded
above by one, we have
$$
\e N_R \leq
\int_{w>0}
\frac{dw}{w} \ 3 \exp(-w \log n/(6c_0)) \
\e 2r(\frac{S_i}{w}) \
$$
$$
= 6 c_1 \int_{w>0} \exp(-w \log n/(6c_0))\  dw
\int_{y>0}       e^{-wy} e^{-c_2 \sqrt{y}} \ dy
$$
$$
< 6 c_1 \int_{y>0} e^{-c_2 \sqrt{y}} \ dy
 \int_{w>0} \exp(-w \log n/(6c_0))\  dw = O\left( \frac{1}{\log n}
 \right).
$$
This completes the proof of (\ref{dist to N 3}).
\hfill \qed



\begin{thebibliography}{99}
\parindent 0pt
\baselineskip
 12pt
\parskip=0pt

\bibitem{aldous mapping} Aldous, D. J. (1983)  Exchangeability and
related topics.  Springer, Lecture Notes in Mathematics,
vol. 1117.

\bibitem{primedw} Arratia, R. (1996) Independence of prime factors:
 total variation and Wasserstein metrics, insertions and deletions,
and the Poisson-Dirichlet process.  Draft, 70 pages.

\bibitem{dimacs} Arratia, R. (1998)
On the central role of scale invariant Poisson processes
on $(0,\infty)$.    Microsurveys in Discrete Probability
(Princeton, NJ, 1997), 21--41, (edited by D. Aldous and J. Propp)
DIMACS Ser. Discrete Math. Theoret. Comut. Sci., 41 (1998), Amer.
Math. Soc., Providence, RI.

\bibitem{tv3} Arratia, R., Barbour, A. D., and Tavar\'e, S. (1992)
Poisson process approximation for the Ewens Sampling Formula.  Ann.
Appl. Probab. {\bf 2}, 519-535.

\bibitem{abt} Arratia, R., Barbour, A. D., and Tavar\'e, S. (2001)
Logarithmic combinatorial structures.  Monograph, in preparation.
Draft version available at
{\tt http://www-hto.usc.edu/books/tavare/ABT/index.html}

\bibitem{AMS1} Arratia, R., Barbour, A. D., and Tavar\'e, S. (1997)
Random combinatorial structures and prime factorizations.  AMS Notices
{\bf 44}, 903-910.


\bibitem{abtlocal} Arratia, R., Barbour, A. D., and Tavar\'e, S. (1999)
On Poisson-Dirichlet limits for random decomposable combinatorial
structures.    Combin., Probab., Comput. {\bf 8} 193-208


\bibitem{T1} Arratia, R., Barbour, A. D., and Tavar\'e, S. (1999)
The Poisson-Dirichlet distribution and the
scale invariant Poisson process.
Combin., Probab.,  Comput. {\bf 8}, 407-416.

\bibitem{four couplings}  Arratia, R., Barbour, A. D., and Tavar\'e, S. (1999)
Expected $l_1$ distance in
Poisson-Dirichlet approximations
for random permutations: a tale of four couplings.
Preprint


\bibitem{AS} Arratia, R., and Stark, D. (1999)
A total variation distance
invariance principle for primes, permutations and
Poisson-Dirichlet. Preprint.

\bibitem{tv1} Arratia, R., and Tavar\'e, S. (1992) The cycle structure
of random permutations. Ann. Probab. {\bf 20}, 1567-1591.

\ignore{
\bibitem{tv2} Arratia, R., and Tavar\'e, S. (1992)
Limit theorems for combinatorial structures via discrete process
approximations.  Random Structures and Algorithms {\bf 3}, 321-345.

\bibitem{tvg} Arratia, R., and Tavar\'e, S. (1994) Independent process
approximation for random combinatorial structures.  Advances in Math.
{\bf 104}, 90-154.
}

\bibitem{course} Arratia, R., and Tavar\'e, S. (1996) Random partitions,
permutations, and primes.  Course lecture notes for Math 533. University
of Southern California, Department of Mathematics.

\bibitem{bach} Bach, E. (1985) Analytic Methods in the Analysis
and Design of Number-theoretic Algorithms.  The MIT Press.

\bibitem{barban} Barban, M. B., and Vinogradov, A. I. (1964)
On the number theoretic basis of the probabilistic
theory of numbers.  Dokl. Akad. Nauk SSSR {\bf 154}, 495-496.


\bibitem{statsci2} Barbour, A. D. (1990)
Comment on ``Poisson approximation and the Chen-Stein method''.
Statistical Science
{\bf 5}, 425-427.

\bibitem{ulam}
Beyer, W., Stein, M., Smith, T., and Ulam, S. (1972)
Metrics in Biology, an Introduction. Los Alamos Scientific Laboratory
report LA-4973.
Reprinted in Analogies between Analogies; the mathematical reports of S.
M. Ulam
and his Los Alamos collaborators, 1990 University of California
Press.

\bibitem{billing large} Billingsley, P. (1972) On the distribution of
large prime factors.  Period. Math. Hungar. {\bf 2}, 283-289.

\bibitem{billingsley-monthly} Billingsley, P. (1973)
Prime numbers and Brownian motion.
Amer. Math. Monthly {\bf 80}, 1099-1115.

\bibitem{billingsley} Billingsley, P. (1974) The 1973 Wald Memorial
Lecture: The probability theory of additive arithmetic functions.
 Ann. Probab. {\bf 2}, 749-
791.


\ignore{\bibitem{buch} Buchstab, A. A. (1937) An asymptotic estimation of a
general number-theoretic function.  Mat. Sbornik (2) {\bf 44},
1239-1246.}

\bibitem{dp86} Diaconis, P., and Pitman, J.  (1986) Permutations, record
values and random measures. Unpublished lecture notes, Dept. Statistics,
U. C. Berkeley.

\bibitem{dickman} Dickman, K. (1930)  On the frequency of numbers
containing prime factors of a certain relative magnitude.  Ark. Math.
Astr. Fys. {\bf 22}, 1-14.

\bibitem{DG} Donnelly, P., and Grimmett, G. (1993) On the asymptotic
distribution of large prime factors. J. London Math. Soc. (2) {\bf 47},
395-404.

\bibitem{dudley} Dudley, R. M. (1989) Real Analysis and Probability.
Wadsworth and Brooks/Cole.

\bibitem{elliott} Elliott, P. D. T. A. (1979) Probabilistic Number Theory
I,   Springer Grundlehren der math. Wissenschaften 239.

\bibitem{galambos} De Koninck, J.-M., and Galambos, J. (1987) The
intermediate prime divisors of integers.  Proc. Amer. Math. Soc. {\bf
101}, 213-216.

\bibitem{feller45} Feller, W. (1945)  The fundamental limit
theorems in probability.  Bull. Amer. Math. Soc., {\bf 51},
800-832.

\bibitem{divisors} Hall, R.R., and Tenenbaum, G. (1988)
Divisors.  Cambridge Tracts in Mathematics {\bf 90},
Cambridge.

\bibitem{hansen} Hansen, J. C. (1994). Order
statistics for decomposable combinatorial structures.
Random Structures and  Algorithms {\bf 5}, 517-533.

\bibitem{ignatov81} Ignatov, Z. (1981)
Point processes generated by order statistics and their
applications. In Point processes and queuing problems (Colloq., Keszthely,
             1978),109--116, North-Holland, Amsterdam-New York

\ignore{\bibitem{ignatov} Ignatov, Z.
(1982) On a constant arising in the asymptotic theory
of symmetric groups, and on Poisson-Dirichlet measures. Theor. Prob.
Appl. {\bf 27}, 136-147.

\bibitem{kac} Kac, M. (1959) Statistical Independence in Probability, Analysis,
and Number Theory.  Carus math. monograph 12, MAA, Wiley.
}

\bibitem{kingman77} Kingman, J.F.C. (1977)
The population structure associated with the Ewens sampling
  formula.
Theor. Pop. Biol. {\bf 11}, 274--283.

\bibitem{kingman} Kingman, J.F.C. (1993) Poisson Processes.  Oxford
Science Publications

\bibitem{knuth} Knuth, D., and Trabb Pardo, L. (1976) Analysis of a
simple factorization algorithm. J. Theoret. Comput. Sci. {\bf 3},
321-348.

\ignore{\bibitem{knuth2e} Knuth, D. (1981) The Art of Computer Programming, v.
2, {\em second edition only}.  Addison Wesley.}

\bibitem{time warps} Kruskal, J., and Sankoff, D. (1983) Time warps,
string edits, and macromolecules:
the theory and practice of sequence comparison.  Addison-Wesley.

\bibitem{kub} Kubilius, J. (1962, translated 1964) Probabilistic Methods
in the Theory of Numbers.   AMS Translations of Mathematical Monographs,
{\bf 11}.

\bibitem{kurtz} Kurtz, T. (1978) Strong approximation theorems for
density dependent Markov chains.   Stochastic Procs. Appls. {\bf 6},
223-240.

\bibitem{levenstein} Levenstein, V. (1965) Binary codes capable of
correction deletions,
insertions, and reversals.  Cybernetics and Control Theory {\bf
10} 707-710 (1996); Russian Doklady Akademii Nauk SSR {\bf 163}
845-848.

\bibitem{fractale} Mend\`es France, M., and Tenenbaum, G. (1993)
Syst\`emes de points, diviseurs, et structure fractale.
Bull. Soc. Math. de France {\bf 121}, 197-225.

\bibitem{philipp} Philipp, W. (1973)
Arithmetic functions and Brownian  motion.
Proc. Sympos. Pure Math {\bf 24} 233-246.

\bibitem{pityor97} Pitman, J., and Yor, M. (1997)
The two-parameter Poisson-Dirichlet distribution
derived from a stable subordinator.  Ann. Probab. {\bf 25}, 855-900.

\bibitem{renyi turan} R\'enyi, A., and Tur\'an, P. (1957) On a
theorem of Erd\H{o}s-Kac.  Acta Arith. {\bf 4}, 71-84.

\bibitem{rio2} Rio, E. (1994) Local invariance principles and their
application to density estimation.
Probab. Th. Related Fields {\bf 98}, 21-45.

\bibitem{rs} Rosser, J. B., and Schoenfeld, L. (1962) Approximate
formulas for some functions of prime numbers.  Illinois J. Math. {\bf
6}, 64-94.

\bibitem{stark th} Stark, D. (1997)
Explicit limits of total variation distance in
 approximations of random logarithmic combinatorial
assemblies by related Poisson processes. Combin., Probab., Comput. {\bf 6}, 87-105.

\bibitem{tenenbaum} Tenenbaum, G. (1995) Introduction to analytic and
probabilistic number theory. Cambridge studies in advanced mathematics,
{\bf 46}.  Cambridge University Press.

\bibitem{crible}  Tenenbaum, G. (1999).
Crible d'\'Eratosth\`ene et mod\`ele de Kubilius,
Number theory in progress, Vol. 2 (Zakopane-Ko\'scielisko, 1997),
1099--1129, de Gruyter, Berlin, 1999.


\bibitem{rate bill} Tenenbaum, G. (2000).
A rate estimate in Billingsley's theorem for the
size distribution of large prime factors.
Quart. J. Math. {\bf 51}, 387-405.

\bibitem{vershik1} Vershik, A.M., and Shmidt, A.A. (1977) Limit measures
arising in the theory of groups, I, Theory Probab. Appl. {\bf 22}, 79-
85.

\bibitem{vershik2} Vershik, A.M., and Shmidt, A.A. (1978) Limit measures
arising in the theory of groups, II, Theory Probab. Appl. {\bf 23}, 36-
49.


\end{thebibliography}
\end{document}